\documentclass[10pt]{amsart}
\usepackage{amsthm}
\usepackage{amsmath}
\usepackage{amssymb}
\usepackage{comment}
\usepackage{enumitem}
\usepackage{amsfonts}
\usepackage{mathtools}
\usepackage[dvipsnames]{xcolor}
\usepackage[hidelinks]{hyperref}
\usepackage{dynkin-diagrams}
\usepackage{dsfont}
\hypersetup{
    colorlinks=true,
    linkcolor=blue,
    urlcolor=red,
    citecolor=Green
    }

\usepackage[utf8]{inputenc}

\advance\paperheight1.00cm
\advance\textheight1.00cm
\advance\vsize1.00cm
\advance\topskip-0.5cm
\advance\voffset-0.5cm
\usepackage{tikz,tikz-cd}
\usepackage{todonotes}
\DeclareTextFontCommand{\emph}{\color{blue}\em}
\newcommand{\bigdot}{\boldsymbol{\cdot}}

\newcommand\precdot{\mathrel{\ooalign{$\prec$\cr
  \hidewidth\raise0.001ex\hbox{$\cdot\mkern0.6mu$}\cr}}}

\newtheorem{theorem}{Theorem}[section]
\newtheorem{proposition}[theorem]{Proposition}

\newtheorem{claim}[theorem]{Claim}
\newtheorem{lemma}[theorem]{Lemma}
\newtheorem{corollary}[theorem]{Corollary}

\theoremstyle{definition}
\newtheorem{ex}[theorem]{Example}
\newtheorem{definition}[theorem]{Definition}

\theoremstyle{remark}
\newtheorem*{remark}{Remark}

\newcommand{\bs}{\symbol{'134}}%

\DeclareMathOperator{\wt}{wt}
\DeclareMathOperator{\codim}{codim}

\DeclareMathOperator{\AD}{AD}
\DeclareMathOperator{\TD}{TD}
\DeclareMathOperator{\ad}{ad}
\DeclareMathOperator{\td}{td}
\DeclareMathOperator{\SL}{SL}
\DeclareMathOperator{\supp}{supp}
\DeclareMathOperator{\Supp}{Supp}
\DeclareMathOperator{\Heart}{Heart}
\newcommand{\R}{\mathcal{R}}

\usepackage[backend=biber, style=alphabetic, maxbibnames=999,
    maxalphanames=999, doi=true, url=false, natbib=true]{biblatex}
\DeclareLabelalphaTemplate{%
  \labelelement{\field[final]{shorthand}}%
  \labelelement{\field[uppercase,strwidth=1,strside=left]{labelname}}%
  \labelelement{\field[strwidth=2,strside=right]{year}\field{extrayear}}%
}
\AtBeginDocument{%
}

\title{Orbit structures and complexity in Schubert and Richardson Varieties}
\author{Yibo Gao}
\address{Beijing International Center for Mathematical Research, Peking University, Beijing 100871, China}
\email{gaoyibo@bicmr.pku.edu.cn}
\author{Reuven Hodges}
\address{Department of Mathematics, University of Kansas, Lawrence, KS 66045, USA}
\email{rmhodges@ku.edu}

\date{\today}

\begin{document}
\pagestyle{plain}
\begin{abstract}
The goal of this paper is twofold. Firstly, we provide a type-uniform formula for the torus complexity of the usual torus action on a Richardson variety, by developing the notion of algebraic dimensions of Bruhat intervals, strengthening a type $A$ result by Donten-Bury, Escobar and Portakal. In the process, we give an explicit description of the torus action on any Deodhar component as well as describe the root subgroups that comprise the component. Secondly, when a Levi subgroup in a reductive algebraic group acts on a Schubert variety, we exhibit a codimension preserving bijection between the Levi-Borel subgroup (a Borel subgroup in the Levi subgroup) orbits in the big open cell of that Schubert variety and torus orbits in the big open cell of a distinguished Schubert subvariety. This bijection has many applications including a type-uniform formula for the Levi-Borel complexity of the usual Levi-Borel subgroup action on a Schubert variety. We conclude by extending the Levi-Borel complexity results to a large class of Schubert varieties in the partial flag variety.
\end{abstract}
\maketitle

\section{Introduction}
\subsection{Group orbits in the full flag variety} Let $G$ be a complex, connected, reductive algebraic group of rank $r$. Fix a maximal torus $T$ in $G$, and a Borel subgroup $B$ of $G$ containing $T$. The homogeneous space $G/B$ is a smooth projective variety known as the \emph{full flag variety}. The study of the flag variety first arose out of the need to formalize and justify the enumerative geometry of H. Schubert as laid out in Hilbert's 15th problem. These varieties are a central theme in much of the mathematics of the 20th century and beyond, with deep and fundamental connections to algebraic geometry, Lie theory, representation theory, algebraic combinatorics, and commutative algebra. 

Of particular importance has been the study of orbits and orbit closures in the flag variety. Famously, the $B$-orbits for the action of $B$ by left translation yield a cellular filtration of $G/B$; these orbits, denoted $X^{\circ}_w$, are referred to as \emph{Schubert cells} and are indexed by elements $w$ of the Weyl group $W$ of $G$. In his seminal work \cite{C94}, C. Chevalley introduced the now ubiquitous Bruhat order to describe the inclusion order of $B$-orbit closures in $G/B$. These orbit closures, or as they are more commonly known, \emph{Schubert varieties} $X_w$ for $w \in W$, are well-studied varieties that boast a rich combinatorial structure that encodes many facets of their geometry. They are known to be normal, Cohen–Macaulay, and have rational singularities~\cite{DeCon-Lak,  RR85, R85}. 

A stratification of $G/B$ via the orbits of the opposite Borel subgroup $B^-$ also exists; the closure of these orbits are the \emph{opposite Schubert varieties} $X^w$ for $w \in W$. The \emph{Richardson variety} $\mathcal{R}_{u,v}$ for $u,v \in W$ is defined to be the intersection $\mathcal{R}_{u,v} := X_v \cap X^u$. Though Richardson varieties have arisen in the literature in a number of guises, initially in the work of W. V. D. Hodge \cite{H42}, they were first introduced in their modern form by D. Kazhdan and G. Lusztig~\cite{KL80}. They were first shown to be irreducible by Deodhar~\cite{deodhar}. Later, R. W. Richardson, in whose honor the varieties are named, studied more general intersections of double cosets and their closures, proving their irreducibility~\cite{R92}. M. Brion proved that Richardson varieties are normal, Cohen–Macaulay, and have rational singularities~\cite{Br02}.

While the Borel subgroup orbits are perhaps the most fundamental class of orbits in the full flag variety, they are not the only orbits that have been explored. The orbits of $T$ in the flag variety were first studied by Klyachko~\cite{K85}. In related work, I. M. Gelfand and V. V. Serganova showed that the strata of the Grassmannian determined by which Pl\"ucker coordinates vanish are indexed by matroids realizable over $\mathbb{C}$~\cite{GS87}. This connection between matroid theory can be extended to torus orbits in the full flag variety via flag matroids~\cite[\S 6.5]{CDMS22}. 

\subsection{Torus orbits in Schubert and Richardson Varieties} It is also instructive to study group orbits and their closures within Schubert and Richardson varieties. Of particular interest in this setting is the statistic on the set of orbits of a reductive algebraic group $H$ known as the $H$-complexity, a notion which we now make precise.

If an algebraic group $H$ acts on a variety $X$ by a morphism of algebraic varieties, we say that $X$ is an $H$-variety. We denote the set of $H$-orbits in $X$ by $\mathcal{O}_{H}(X)$.

\begin{definition}\label{def:complexity}
Let $H$ be a reductive algebraic group and $B_H$ a Borel subgroup of $H$.  The \emph{$H$-complexity} of an $H$-variety $X$, denoted $c_{H}(X)$, is the minimum codimension of a $B_H$-orbit in $X$.
\end{definition}

The normal $H$-varieties with $H$-complexity equal to $0$ are the \emph{$H$-spherical varieties}. Spherical varieties generalize several important classes of algebraic varieties including toric varieties, projective rational homogeneous spaces and symmetric varieties. The birational models of a given spherical variety are classified via Luna-Vust theory using colored fans~\cite{LV83,K91}, a generalization of the fans which classify toric varieties. In type $A$, Luna classified spherical homogeneous spaces by homogeneous spherical data~\cite[\S 2, Theorem~2]{L01}.

Unless otherwise stated, in this paper the action of any subgroup of $G$ on any subvariety of $G/B$ will always be left translation. We will refer to this as the \emph{usual action}. 

For any $w \in W$, both $X_w$ and $X^w$ are $T$-varieties for the usual action. We denote the Bruhat order on $W$ by $\leq$, writing $u \leq v$ if and only if $X_u \subseteq X_v$. For $u \leq v \in W$, $\mathcal{R}_{u,v}$ is nonempty and is a normal $T$-variety for the usual action. The maximal torus $T$ is a reductive algebraic group whose only Borel subgroup is $T$ itself. Hence, a normal $T$-variety with $T$-complexity equal to $0$ is a $T$-spherical variety (that is, a toric variety).

P. Karuppuchamy provided a succinct, root-system uniform classification of the toric Schubert varieties~\cite{K13}. E. Tsukerman and L. Williams studied Bruhat interval polytopes and introduced their generalization to $G/P$, identifying these polytopes with moment-map images of Richardson and projected Richardson varieties~\cite{TW15}. This connection was expanded on by E. Lee, M. Masuda, and S. Park, yielding a classification of toric Richardson varieties in $G/B$ where $G$ is of Dynkin-type $A$~\cite{LMP21b}. C. Gaetz obtained further results on lower Bruhat interval polytopes and generic torus orbit closures in Schubert varieties~\cite{Gaetz-Bruhat-interval-polytope}. In \cite{LMP21a}, the same authors give a classification of $T$-complexity 1 Schubert varieties in $G/B$ when $G$ is of Dynkin-type $A$. A classification of toric Richardson varieties in any full flag variety was provided by M. Can and P. Saha in \cite{mahir-toric}. 

Our first main result continues this line of research, giving a root-system uniform combinatorial formula for the $T$-complexity of a Richardson variety in any full flag variety. We define the \emph{algebraic dimension} $\ad(u,v)$ of a Bruhat interval $[u,v]$ in Section~\ref{sec:torus-complexity-Richardson}. There, we identify $\ad(u,v)$ with the dimension of the corresponding Bruhat interval polytope, which, together with the moment-map results of Tsukerman--Williams \cite{TW15}, gives the maximal dimension of a $T$-orbit in $\mathcal{R}_{u,v}$. We then prove a Coxeter-theoretic formula for $\ad(u,v)$ in terms of maximal toric subintervals.

\begin{theorem}\label{thm:torus-complexity}
For $u\leq v \in W$, the $T$-complexity of the Richardson variety equals \[c_T(\mathcal{R}_{u,v})=\ell(v)-\ell(u)-\ad(u,v).\] Moreover, we have that \[\ad(u,v)=\max_{u\leq w\leq v}\ell(v)-\ell(w)\text{ where }\mathcal{R}_{w,v}\text{ is toric}.\]
\end{theorem}

The first formula in Theorem~\ref{thm:torus-complexity} generalizes the type-$A$ Richardson-complexity formula in \cite[Corollaries~4.15--4.16]{escobar-complexity} to other Lie types. As pointed out by an anonymous referee, the first part of Theorem~\ref{thm:torus-complexity} is implicit in \cite{TW15}. 
Our independent approach has additional benefits: first, we are able to describe the standard torus dimensions in any Deodhar components (Theorem~\ref{thm:torus-dimension-richardson}); second, our results on $\AD(u,v)$ along the way carry verbatim to any chosen Coxeter datum, even where the theory of Bruhat interval polytopes is nonexistent. 

In the case of $u=\mathrm{id}$, Theorem~\ref{thm:torus-complexity} simplifies to a simple, type-uniform combinatorial formula for the torus complexity of a Schubert variety.
\begin{corollary}\label{cor:torus-complexity-Schubert}
For $w\in W$, the $T$-complexity of the Schubert variety $X_w$ equals
\[
c_T(X_w)=\ell(w)-\supp(w).
\]
\end{corollary}

We also refine the maximal orbit-dimension calculation by describing the torus weights on each Deodhar component. We state this result using the notation of Section~\ref{sub:torus-complexity}; in particular, $D_{\mathbf{u},\mathbf{v}}$ denotes a Deodhar component and $\TD(\mathbf{u},\mathbf{v})$ denotes the real span of the roots corresponding to its non-fixed Deodhar parameters.

\begin{theorem}\label{thm:torus-dimension-richardson}
Fix $u\leq v$ and fix a reduced word $\mathbf{v}$ of $v$. Let $\mathbf{u}$ be a distinguished subword of $\mathbf{v}$ for $u$, and let $\mathbf{u}_+$ be the positive distinguished subword for $u$ in $\mathbf{v}$. For $x\in D_{\mathbf{u},\mathbf{v}}$, let $g$ be its standard form. Then
\[
\dim Tx=\dim_{\mathbb{R}}\mathrm{span}\{\beta_k\mid k\in J_{\mathbf{u}}^{\circ}\cup J_{\mathbf{u}}^-,\ g^{(k)}\ne0\}.
\]
In particular, $\dim Tx=\td(\mathbf{u},\mathbf{v})$ for a general point $x$. Moreover, we have
\[
\TD(\mathbf{u},\mathbf{v})\subseteq \TD(\mathbf{u}_+,\mathbf{v})=\AD(u,v).
\]
\end{theorem}

Theorem~\ref{thm:torus-dimension-richardson} gives a Deodhar stratum refinement of the maximal $T$-orbit dimension underlying the first formula in Theorem~\ref{thm:torus-complexity}. It extends the torus-weight calculation of \cite[Theorem~4.7]{mahir-toric} from positive distinguished subwords to arbitrary distinguished subwords, and identifies the maximal such weight span with $\AD(u,v)$.

\subsection{Levi-Borel subgroup orbits in Schubert Varieties} We now turn our attention to the orbits of Levi-Borel subgroups (defined below) in a Schubert variety. Our work uncovers a deep structural relationship between the Levi-Borel orbits in the big cell of a Schubert variety and the $T$-orbits in the big cell of a distinguished Schubert subvariety. Specifically, when a Levi subgroup acts on a Schubert variety we exhibit a codimension-preserving bijection between these sets of orbits.

Our choice of $T$ and $B$ determine the \emph{root system} $\Phi$ and a set of \emph{simple roots} $\Delta=\{ \alpha_1,\ldots,\alpha_r\}$, respectively. The Weyl group $W$ of $G$, is generated by the set of simple reflections $S=\{s_i:=s_{\alpha_i}\mid\alpha_i\in\Delta\}$. For $I\!\subseteq\!\Delta$, $W_I$ is the subgroup of $W$ generated by $\{s_{\alpha_i}\mid\alpha_i\in I\}$. 

The \emph{standard parabolic subgroups} of $G$ containing $B$ are indexed by subsets of $\Delta$. The standard parabolic subgroup associated to $I$ is $P_I = B W_I B$ with Levi decomposition \[P_I = L_I \ltimes U_I,\] where $L_I$ is a reductive subgroup called a \emph{Levi subgroup}, and $U_I$ is the unipotent radical of $P_I$. Define $B_{L_I} := L_I \cap B$. Then $B_{L_I}$ is a Borel subgroup of $L_I$, and we shall refer to such subgroups as \emph{Levi-Borel subgroups}. Our second main result concerns the orbits of Levi-Borel subgroups in a Schubert cell.

\begin{theorem}
\label{theorem:torusLeviBorelBijection}
Let $w \in W$, $I \subseteq \Delta$, and let $w= { _Iw} \prescript{I}{}{w}$ be the left parabolic decomposition of $w$ with respect to $I$. The map $\mathfrak{O}\!:\!\mathcal{O}_T(X^{\circ}_{\prescript{I}{}{w}})\!\rightarrow\!\mathcal{O}_{B_{L_I}}(X_w^{\circ})$ given by $\Theta\!\mapsto \!B_{L_I} {_I}{w} x$, 
where $x$ is any point in $\Theta$, is a surjection. If $L_I$ acts on $X_w$, then $\mathfrak{O}$ is a codimension preserving bijection. That is, \[ \dim(X_w^{\circ}) - \dim(\mathfrak{O}(\Theta)) = \dim(X^{\circ}_{\prescript{I}{}{w}}) - \dim(\Theta), \]
for every $\Theta \in \mathcal{O}_T(X^{\circ}_{\prescript{I}{}{w}})$.
\end{theorem}

Theorem~\ref{theorem:torusLeviBorelBijection} allows us to intertwine the study of Levi-Borel orbits with the vast literature on torus orbits and varieties, including our own Corollary~\ref{cor:torus-complexity-Schubert}, and has myriad applications which we now detail. Our first application is Proposition~\ref{proposition:BLICodim} which provides a lower bound on the codimension of a $B_{L_I}$-orbit in $X_w$. In the case where the Levi subgroup $L_I$ acts on the Schubert variety $X_w$ this leads to a closed formula for the minimal codimension of a $B_{L_I}$-orbit in $X_w$.

While $B_{L_I}$ acts on any $X_w$, since $B_{L_I} \subseteq B$, the same is not true for $L_I$. The stabilizer of $X_w$ in $G$ for the usual action is the standard parabolic subgroup $P_{{\mathcal D}_L(w)}$ \cite[Lemma 8.2.3]{BL00}, where ${\mathcal D}_L(w)$ is the left descent set of $w$ defined in Section~\ref{sec:prelim}. Thus, the Levi subgroups $L_I \leq P_I \leq P_{{\mathcal D}_L(w)}$ for $I \subseteq {\mathcal D}_L(w)$ are a family of reductive algebraic groups which act on $X_w$. Our third main result is a root-system uniform, combinatorial formula for the $L_I$-complexity of any Schubert variety which is an $L_I$-variety in any full flag variety.  

\begin{theorem}\label{theorem:generalTypeLeviComplexity}
Let $w \in W$ and suppose $L_I$ acts on the Schubert variety $X_w$ (equivalently, $I \subseteq \mathcal{D}_L(w)$). If $w= { _Iw} \prescript{I}{}{w}$ is the left parabolic decomposition of $w$ with respect to $I$, then \[c_{L_I}(X_w) = \ell(\prescript{I}{}{w}) - \supp(\prescript{I}{}{w}).\]
\end{theorem}

Theorem~\ref{theorem:generalTypeLeviComplexity} is a refinement and generalization of numerous earlier results. The second author and A. Yong initiated a study of $L_I$-complexity $0$ Schubert varieties in \cite{HY22}, and therein conjectured a classification in terms of spherical elements of a Coxeter group. Subsequently, both authors of this paper and A. Yong proved this conjecture first for $G$ of Dynkin-type $A$~\cite{GHY23}, and later for any full flag variety \cite{GHY24} (the latter was proved contemporaneously via different methods by M. Can and P. Saha \cite{CS23-1}). Using work of R.~S. Avdeev and A.~V. Petukhov \cite{AP14}, the classification of $L_I$-complexity $0$ Schubert varieties may be interpreted as a generalization of results of P. Magyar, J. Weyman, and A. Zelevinsky \cite{MWZ99} and J. Stembridge~\cite{S01, S03} on spherical actions on (products of) flag varieties. The interested reader may see \cite[Theorem~2.4]{HY22} for the details. 

Our next step, with work in progress, is to use Theorem~\ref{theorem:torusLeviBorelBijection} to index the orbits of Levi-Borel subgroups in the Schubert variety itself and understand the inclusion order on their orbit closures. Throughout the years, the study of such Bruhat orders has led to much fruitful research. One related example, is that of the orbits of spherical Levi subgroups in Dynkin type $A$. These orbits are indexed by $(p,q)$-clans~\cite{MO90,Y97}, and their inclusion order is described via clan combinatorics~\cite{W16}.

As a final application of Theorem~\ref{theorem:torusLeviBorelBijection}, in Theorem~\ref{thm:sphericalityTransfer} we compute the $L_I$-complexity of a large class of Schubert varieties in the partial flag variety. This leads to Corollary~\ref{cor:toricClassification}, which is a root-system uniform, combinatorial classification of toric Schubert varieties in any partial flag variety; as far as the authors are aware this classification is new.

The outline of this paper is as follows. In Section~\ref{sec:prelim}, we introduce the necessary background, including Schubert geometry and Coxeter groups. In Section~\ref{sec:torus-complexity-Richardson}, we fully develop the notion of algebraic dimensions and establish Theorem~\ref{thm:torus-complexity}. In Section~\ref{sec:levi-complexity}, we prove Theorem~\ref{theorem:torusLeviBorelBijection} and Theorem~\ref{theorem:generalTypeLeviComplexity} with one of our main tools being an equivariant isomorphism that allows us to use results in Section~\ref{sec:torus-complexity-Richardson}. Section~\ref{sec:levi-complexity} concludes with an extension of the Borel-Levi subgroup complexity results to Schubert varieties in partial flag varieties.

\section{Preliminaries}\label{sec:prelim}

\subsection{Combinatorics of the Weyl group}
Consider the following data associated to elements of the Weyl group $W$. For $w\in W$, its \emph{Coxeter length} is the smallest $\ell$ such that $w$ can be written as product of $\ell$ simple reflections. Such an expression is called a \emph{reduced word} or a \emph{reduced expression} of $w$. 

For $u,v \in W$, the product $u v$ is said to be \emph{length additive} if $\ell(uv)=\ell(u) + \ell(v)$. Let $W^I$ be the set of minimal coset representatives in $W$ of $W/W_I$. In the same way, let $\prescript{I}{}{W}$ be the set of minimal coset representatives in $W$ of $W_I \bs W$. Given $w\in W$ and $I\subseteq\Delta$, $w$ has a unique \emph{right parabolic decomposition} $w=w^I w_I$ which is length-additive with $w_I\in W_I$ and $w^I\in W^I$. Similarly, $w$ has a \emph{left parabolic decomposition} $w= { _Iw} \prescript{I}{}{w}$ which is length-additive with $_Iw\in W_I$ and $^Iw\in \prescript{I}{}{W}$.

The \emph{support} of $w$ is \[\Supp(w)=\{\alpha_i\in\Delta\:|\: s_i\text{ appears in any/all reduced words of }w\}.\]
The cardinality of $\Supp(w)$ is written as $\supp(w)=|\Supp(w)|$. 
For $w\in W$, its \emph{left descent set} and the \emph{right descent set} are \[{\mathcal D}_L(w)=\{\alpha_i\in \Delta\:|\:\ell(s_{i}w)<\ell(w)\},\quad {\mathcal D}_R(w)=\{\alpha_i\in \Delta\:|\:\ell(ws_i)<\ell(w)\},\] respectively. And its \emph{left inversion set} and the \emph{right inversion set} are \[{\mathcal I}_L(w)=\{\alpha\in \Phi^+\:|\:w^{-1}(\alpha)\in\Phi^-\},\quad {\mathcal I}_R(w)=\{\alpha\in \Phi^+\:|\:w(\alpha)\in\Phi^-\},\] respectively, with $|\mathcal{I}_L(w)|=|\mathcal{I}_R(w)|=\ell(w)$. There are equivalent ways to write down the inversion sets. For example, \[\mathcal{I}_L(w)=\Phi^+\cap w(\Phi^-)=\{\alpha\in\Phi^+\:|\: \ell(s_{\alpha}w)<\ell(w)\}.\]
Given a length additive product $uv$, \cite[Ch. VI, \S 1, Cor. 2 of Prop. 17]{B02} proves that
\begin{equation}
\label{eq:lengthaddinvset}
{\mathcal I}_L(uv) = {\mathcal I}_L(u) \sqcup u({\mathcal I}_L(v)).
\end{equation}
The \emph{(strong) Bruhat order} is a partial order on $W$ that is generated by \[w<ws_{\alpha}\text{ for }\alpha\in\Phi^+,\text{ if }\ell(w)<\ell(ws_{\alpha}).\]
Write $[u,v]:=\{w\in W\:|\: u\leq w\leq v\}$ for a \emph{Bruhat interval}. It is well-known (see \cite[Lemma 2.7.3]{Bjorner.Brenti}) that any interval $[u,v]$ of rank $2$, i.e. $\ell(v)-\ell(u)=2$, has $4$ elements forming a ``diamond". 
Here are some properties of the Bruhat order. See \cite[Section 2.2]{Bjorner.Brenti}.
\begin{proposition}[Subword property]\label{prop:subword-property}
Fix a reduced word $w=s_{i_1}\cdots s_{i_{\ell}}$. Then $u\leq w$ in the Bruhat order if and only if there exists a subexpression of $s_{i_1}\cdots s_{i_{\ell}}$ that equals $u$.
\end{proposition}
\begin{proposition}[Lifting property]\label{prop:lifting}
Suppose $u<w$ and $s\in \mathcal{D}_R(w)\setminus\mathcal{D}_R(u)$, then $u\leq ws$ and $us\leq w$. 
\end{proposition}

\subsection{Root subgroups}\label{sub:root-subgroups}
For $\alpha \in \Phi$, let $U_{\alpha}$ be the root subgroup corresponding to $\alpha$. Given $w \in W$, define $U_w := U \cap w U^- w^{-1}$, 
where $U$ and $U^-$ are the unipotent parts of $B$ and $B^- := w_0 B w_0$, respectively. A sequence of subgroups $H_1,\ldots,H_k$ in an algebraic group $H$ is said to \emph{directly span} $H$ if the product morphism $H_1 \times \cdots \times H_k \rightarrow H$ is a bijection. The subgroup $U_w$ of $U$ is closed and normalized by $T$. Hence \cite[\S 14.4]{B91} implies that $U_w$ is directly spanned (in any order) by the root subgroups contained in it. Thus
\[U_w := U \cap w U^- w^{-1} = \left( \prod_{\alpha \in \Phi^+} U_{\alpha} \right) \bigcap \left( w \prod_{\alpha \in \Phi^-} U_{\alpha} w^{-1} \right) = \left( \prod_{\alpha \in \Phi^+} U_{\alpha} \right) \bigcap \left( \prod_{\alpha \in \Phi^-} U_{w(\alpha)} \right) ,\]
where the second equality is \cite[\S 14.4]{B91} applied to $U$ and $U^-$, and the last equality is \cite[Part II, 1.4(5)]{J03}. Thus $U_{\alpha}$ is contained in $U_w$ if and only if $\alpha \in {\mathcal I}_L(w)$. Hence
\begin{equation}
\label{eq:formofUw}
U_w = \prod_{\alpha \in {\mathcal I}_L(w)} U_{\alpha}.
\end{equation}

Let $w = ad$ be the left parabolic decomposition of $w$ with respect to $I$, where $a\in W_I$ and $d\in\prescript{I}{}{W}$. Define $V_{d} := a U_d a^{-1}$.

\begin{lemma}\label{lemma:directlyspannedUw}
The group $V_{d}$ is a closed subgroup of $U_w$ normalized by $T$. Indeed, $U_w$ is directly spanned in any order by the subgroups $U_a$ and $V_{d}$, that is, $U_w = U_a V_{d} = V_d U_a$.
\end{lemma}
\begin{proof} 
The group $V_d$ is closed since $U_d$ is closed. Now, let $t \in T$. Since $a \in W$, we have $t^{-1}a = a t^{\prime}$ for some $t^{\prime} \in T$. Hence $a^{-1}t=(t^{\prime})^{-1}a^{-1}$. These results, along with the fact that $U_d$ is normalized by $T$, yield 
\[
    t^{-1} V_d t = t^{-1}aU_d a^{-1} t = t^{-1}a U_d a^{-1} t = a t^{\prime} U_d (t^{\prime})^{-1}a^{-1} = a U_d a^{-1} = V_d,
\]
allowing us to conclude that $V_d$ is normalized by $T$. For the remaining claims, 
\begin{align*}
U_w = \prod_{\alpha \in {\mathcal I}_L(ad)} U_{\alpha} &= \left( \prod_{\alpha \in {\mathcal I}_L(a)} U_{\alpha} \right) \left(  \prod_{\alpha \in a({\mathcal I}_L(d))} U_{\alpha} \right) \\
& = U_a \left( \prod_{\alpha \in {\mathcal I}_L(d)} U_{a(\alpha)} \right) = U_a \left( a U_d a^{-1} \right) = U_a V_d \\
\end{align*}
where the first equality is \eqref{eq:formofUw}, the second is \eqref{eq:lengthaddinvset} combined with the length additivity of the left parabolic decomposition, and the fourth is \cite[Part II, 1.4(5)]{J03} and \eqref{eq:formofUw}. In the same way, one can show $U_w = V_d U_a$. We conclude that $U_w$ is directly spanned in any order by its subgroups $U_a$ and $V_d$.
\end{proof}

\section{The torus complexity of Richardson varieties}\label{sec:torus-complexity-Richardson}
In this section, we establish a formula for the complexity of the usual torus $T$ action on any Richardson variety $\mathcal{R}_{u,v}$ where $u\leq v$ in the Bruhat order.
\begin{definition}\label{def:Bruhat-graph}
The \emph{(undirected) Bruhat graph} on a Weyl group $W$ is the graph $\Gamma$ with vertex set $W$ and edges $w\sim s_{\alpha}w$ for $\alpha\in\Phi^+$. For each edge $w\sim s_{\alpha}w$, we say that it has \emph{label} or \emph{weight} $\alpha$, and write $\wt(w,s_{\alpha}w)=\alpha$. For $u\leq v$, let $\Gamma(u,v)$ be the Bruhat graph restricted to the vertex set $[u,v]$.
\end{definition}
\begin{definition}
For $u\leq v$, let $\AD(u,v)$ be the $\mathbb{R}$-span of all edge labels in $\Gamma(u,v)$, i.e. \[\AD(u,v)=\mathrm{span}_{\mathbb{R}}\{\wt(x,y)\:|\: u\leq x<y\leq v\}.\]
Let $\ad(u,v)=\dim\AD(u,v)$ be the \emph{algebraic dimension} of the Bruhat interval $[u,v]$.
\end{definition}

We now compare $\AD(u,v)$ with the affine span of the Bruhat interval polytope of Tsukerman--Williams. Let $\rho=\frac12\sum_{\alpha\in\Phi^+}\alpha$. Then $2\rho\in X^*(T)$ is a regular dominant character. Define
\[
Q_{u,v}:=\mathrm{Conv}\{z(2\rho) \mid u\leq z\leq v\}\subset X^*(T)_{\mathbb{R}}.
\]
This is twice the full flag polytope in \cite[Definition~7.8]{TW15}.

\begin{proposition}\label{prop:AD-BIP}
For $u\leq v$, we have
\[
\mathrm{span}_{\mathbb{R}}\{q-u(2\rho) \mid q\in Q_{u,v}\}=\AD(u,v).
\]
In particular,
\[
\dim Q_{u,v}=\ad(u,v).
\]
\end{proposition}

\begin{proof}
Let
\[
V:=\mathrm{span}_{\mathbb{R}}\{q-u(2\rho) \mid q\in Q_{u,v}\}.
\]
We first show that $\AD(u,v)\subseteq V$. Let $x<y=s_{\alpha}x$ be an edge of $\Gamma(u,v)$, where $\alpha\in\Phi^+$. Then
\[
y(2\rho)-x(2\rho)=s_{\alpha}x(2\rho)-x(2\rho)
=-\langle x(2\rho),\alpha^\vee\rangle\alpha
=-\langle 2\rho,x^{-1}(\alpha^\vee)\rangle\alpha.
\]
Since $x<s_{\alpha}x$, we have $x^{-1}(\alpha)\in\Phi^+$. Since $2\rho$ is regular dominant, the scalar $\langle 2\rho,x^{-1}(\alpha^\vee)\rangle$ is nonzero. Therefore the label $\alpha=\wt(x,y)$ lies in $V$. Since $\AD(u,v)$ is spanned by such labels, $\AD(u,v)\subseteq V$.

Conversely, let $z\in [u,v]$. Choose a saturated chain
\[
u=w^{(0)}\lessdot w^{(1)}\lessdot \cdots \lessdot w^{(m)}=z.
\]
For each $i$, write $w^{(i+1)}=s_{\alpha_i}w^{(i)}$, so that
\[
\alpha_i=\wt(w^{(i)},w^{(i+1)}).
\]
By the calculation above,
\[
w^{(i+1)}(2\rho)-w^{(i)}(2\rho) \in \mathbb{R}\alpha_i\subseteq \AD(u,v).
\]
Thus
\[
z(2\rho)-u(2\rho)=\sum_{i=0}^{m-1}\left(w^{(i+1)}(2\rho)-w^{(i)}(2\rho)\right)\in \AD(u,v).
\]
Now take any $q\in Q_{u,v}$. Since $Q_{u,v}$ is the convex hull of the points $z(2\rho)$ for $z\in [u,v]$, we may write $q=\sum_z c_z z(2\rho)$, where $c_z\in \mathbb{R}_{\geq 0}$ and $\sum_z c_z=1$. Then $q-u(2\rho)=\sum_z c_z(z(2\rho)-u(2\rho))$. We have shown that each vector $z(2\rho)-u(2\rho)$ lies in $\AD(u,v)$, so $q-u(2\rho)\in \AD(u,v)$. Thus $V\subseteq \AD(u,v)$, hence $V=\AD(u,v)$. Since $V$ is the linear space parallel to the affine span of $Q_{u,v}$, the equality of dimensions follows.
\end{proof}

\begin{corollary}\label{cor:max-t-orbit-ad}
For $u\leq v$, the maximum dimension of a $T$-orbit in $\mathcal{R}_{u,v}$ is $\ad(u,v)$.
\end{corollary}

\begin{proof}
The simply connected cover of the derived subgroup of $G$ induces an isomorphism of flag varieties and identifies torus orbits. Passing from $\rho$ to $2\rho$ doubles the moment map and its polytopes, so the moment-map results of \cite{TW15} apply with this normalization. Let $\mu:G/B\to X^*(T)_{\mathbb{R}}$ denote the moment map of \cite[Definition~6.18]{TW15} associated to the projective embedding determined by $2\rho$. Let $x$ be a point in the totally positive part of the corresponding Richardson cell. By \cite[Lemma~7.9]{TW15}, we have $\mu(Tx)=\operatorname{int}(Q_{u,v})$ and
\[
\mu(\overline{T x})=Q_{u,v}.
\]
By \cite[Theorem~6.19]{TW15}, the $T$-orbits in $\overline{T x}$ are in dimension-preserving correspondence with the open faces of $Q_{u,v}$. In particular, $T x$ has dimension $\dim Q_{u,v}$, so $\mathcal{R}_{u,v}$ contains a $T$-orbit of dimension $\dim Q_{u,v}$.

Conversely, for any $y\in\mathcal{R}_{u,v}$, the orbit closure $\overline{T y}$ is contained in $\mathcal{R}_{u,v}$. By \cite[Remark~7.11]{TW15}, we have $\mu(\mathcal{R}_{u,v})=Q_{u,v}$, and therefore
\[
\mu(\overline{T y})\subseteq Q_{u,v}.
\]
Again by \cite[Theorem~6.19]{TW15},
\[
\dim T y=\dim \mu(\overline{T y})\leq \dim Q_{u,v}.
\]
Thus the maximum dimension of a $T$-orbit in $\mathcal{R}_{u,v}$ is $\dim Q_{u,v}$. By Proposition~\ref{prop:AD-BIP}, this equals $\ad(u,v)$.
\end{proof}

Corollary~\ref{cor:max-t-orbit-ad} supplies the orbit-dimension input for the first assertion of Theorem~\ref{thm:torus-complexity}. It remains to prove the maximal toric subinterval formula for $\ad(u,v)$. This is done in subsection~\ref{sub:algebraic-dimension}. In subsection~\ref{sub:torus-complexity}, we refine the global torus-orbit calculation by studying the torus action on individual Deodhar components.

\subsection{The algebraic dimension of Bruhat intervals}\label{sub:algebraic-dimension}
We first discuss some useful tools.
\begin{lemma}\cite[Lemma 3.1]{dyer-bruhat-graph}\label{lem:Dyer-rank2}
For $\alpha,\beta,\gamma,\tau\in\Phi^+$, if $s_\alpha s_\beta=s_\gamma s_\tau\neq1$, then $\mathrm{span}_{\mathbb{R}}(\alpha,\beta)=\mathrm{span}_{\mathbb{R}}(\gamma,\tau)$ is $2$-dimensional.
\end{lemma}
We will be using this lemma multiple times, primarily on Bruhat intervals of rank $2$.

\begin{lemma}\label{lem:strong-weak-diamond}
For $\alpha_i\in\Delta$ and $\alpha\in\Phi^+$, if $w$ covers both $ws_i\neq s_{\alpha}w$, then both $ws_i$ and $s_{\alpha}w$ cover $s_{\alpha}ws_i$. Dually, if $w$ is covered by both $ws_i\neq s_{\alpha}w$, then $s_{\alpha}ws_i$ covers $ws_i$ and $s_{\alpha}w$.
\end{lemma}
\begin{proof}
For the first statement, it suffices to show that $s_{\alpha}w$ has a right descent at $\alpha_i$. If not, by the lifting property (Proposition~\ref{prop:lifting}) on $w>s_{\alpha}w$, we must have $ws_i\geq s_{\alpha}w$, which is a contradiction since they are different elements of the same length. 

For the dual statement, it suffices to show that $s_{\alpha}w$ does not have a right descent at $\alpha_i$. Assume the opposite. By the first statement on $s_{\alpha}w$, $ws_i$ must be covered by both $w$ and $s_{\alpha}ws_i$, contradicting $w<ws_i$.
\end{proof}

We now describe multiple ways to simplify the computation of $\AD(u,v)$.
\begin{lemma}\label{lem:AD-spanned-by-covers}
$\AD(u,v)$ is spanned by the labels of all cover relations incident to $u$ inside $[u,v]$, and dually, by all cover relations incident to $v$ inside $[u,v].$
\end{lemma}
\begin{proof}
We first show that $\AD(u,v)$ is spanned by all labels of cover relations in $[u,v]$. We say that an edge $x\sim y$ in the Bruhat graph has length $|\ell(x)-\ell(y)|$ and it suffices to show that a label of length greater than $1$ can be linearly generated by edge labels with strictly smaller length. By Proposition 3.3 of \cite{dyer-bruhat-graph}, for any edge $x_1\sim x_6$ such that $\ell(x_1)<\ell(x_6)-1$ in the Bruhat graph, there are $x_2,x_3,x_4,x_5\in[x_1,x_6]$ such that $x_1$ is connected to $x_i$, $x_i$ is connected to $x_j$ and $x_j$ is connected to $x_6$ for all $i\in\{2,3\}$ and $j\in\{4,5\}$. Applying Lemma~\ref{lem:Dyer-rank2} to the reflections corresponding to the edge labels of $x_1\sim x_2\sim x_4\sim x_6$ and $x_1\sim x_6$, we have $\wt(x_1,x_6)\in\mathrm{span}_{\mathbb{R}}\{\wt(x_1,x_2),\wt(x_2,x_4)\}$.

Next, let $V$ be the vector space spanned by the labels of all cover relations incident to $u$, and we need to show that $\wt(x,y)\in V$ for all $x\lessdot y$ inside $[u,v]$. We use induction on $\ell(x)$, where the base case $\ell(x)=\ell(u)$ is given. If $\ell(x)>\ell(u)$, find $u\leq z\lessdot x$ and let $x'\neq x$ be the unique element such that $z\lessdot x'\lessdot y$.  By Lemma~\ref{lem:Dyer-rank2} on $[z,y]$, $\wt(x,y)\in\mathrm{span}_{\mathbb{R}}\{\wt(z,x),\wt(z,x')\}\subset V$. So the induction step is complete. 
\end{proof}

\begin{proposition}\label{prop:AD-recursion}
Let $\alpha_i$ be a right descent of $v$. We have the following recursion: \[\AD(u,v)=\begin{cases}
    \AD(us_i,vs_i) &\text{ if } us_i<u,\\
    \AD(u,vs_i)+\mathbb{R}\cdot\wt(u,us_i) &\text{ if }us_i>u.
\end{cases}\]
\end{proposition}
\begin{proof}
We consider these two cases separately.

\textbf{Case 1:} $vs_i<v$ and $us_i<u$. We use induction on $\ell(v)-\ell(u)$, with the base case $v=u$ being clear. First, we show $\AD(u,v)\subseteq\AD(us_i,vs_i)$. By Lemma~\ref{lem:AD-spanned-by-covers}, it suffices to show that for all $v\gtrdot s_{\alpha}v\in[u,v]$, $\alpha\in\AD(us_i,vs_i)$. If $s_{\alpha}v\neq vs_i$, by Lemma~\ref{lem:strong-weak-diamond}, $s_{\alpha}v$ has a right descent at $s_i$. Now $v\geq s_{\alpha}v\geq u$ and all three have right descents at $s_i$, so $vs_i\geq s_{\alpha}vs_i\geq us_i$. Thus, $s_{\alpha}vs_i$ and $vs_i$ both lie in $[us_i,vs_i]$, meaning that $\alpha=\wt(s_{\alpha}vs_i,vs_i)\in\AD(us_i,vs_i)$. 

Now assume $vs_i=s_{\alpha}v$. Then $vs_i\neq u$ because of right descents. So $vs_i>u$. Choose any $z$ such that $vs_i\gtrdot z\geq u$. Suppose that the rank $2$ interval $[z,v]$ contains $v\gtrdot vs_i,w\gtrdot z$. Let $\beta=\wt(z,vs_i)\in\AD(u,vs_i)\subseteq\AD(us_i,vs_i)$ and $\gamma=\wt(z,w)$. By Lemma~\ref{lem:strong-weak-diamond}, $w$ has a right descent at $s_i$. So $v>w>u$ and $vs_i>ws_i>us_i$. By induction hypothesis on $[u,w]$, $\gamma\in\AD(u,w)=\AD(us_i,ws_i)\subseteq\AD(us_i,vs_i)$. By Lemma~\ref{lem:Dyer-rank2}, $\alpha\in\mathrm{span}_{\mathbb{R}}\{\beta,\gamma\}\subseteq\AD(us_i,vs_i)$ as desired.

Example~\ref{ex:AD-induction-example} shows a specific scenario of this situation.

Dually, the $\supseteq$ direction is proved via the exact same argument by arguing from the cover relations incident to $us_i$ inside $[us_i,vs_i]$.

\textbf{Case 2:} $vs_i<v$ and $us_i>u$. For the $\subseteq$ direction, take any $\alpha$ such that $s_{\alpha}u\gtrdot u$ and $s_{\alpha}u\leq v$. If $s_{\alpha}u=us_i$, then $\alpha=\wt(u,us_i)$ belongs to the right hand side. If $s_{\alpha}u\nleq us_i$, by Lemma~\ref{lem:strong-weak-diamond}, $s_{\alpha}u$ does not have a right descent at $s_i$. The lifting property (Proposition~\ref{prop:lifting}) on $s_{\alpha}u<v$ gives $s_{\alpha}u\leq vs_i$ so $\alpha\in\AD(u,vs_i)$ as desired. For the $\supseteq$ direction, we just need $\wt(u,us_i)\in\AD(u,v)$ which is true as $us_i\leq v$ by the lifting property.
\end{proof}

\begin{ex}\label{ex:AD-induction-example}
Let $u=1324$, $v=3412$ and $i=2$. According to the first case of Proposition~\ref{prop:AD-recursion}, $\AD(1324,3412)=\AD(1234,3142)$. Both intervals are shown in Figure~\ref{fig:3412-lower-interval}. In particular, the rank sizes of $[1324,3412]$ are $1,4,4,1$ while the rank sizes of $[1234,3142]$ are $1,3,3,1$. The covers from the maximum in these two intervals are highlighted. The weights from $3412$ are $e_1-e_3$, $e_2-e_3$, $e_1-e_4$ and $e_2-e_4$ while the weights from $3142$ are $e_1-e_3$, $e_2-e_3$ and $e_2-e_4$ (from left to right). The same linear space is spanned by these two sets of weights.

\begin{figure}[!ht]
\centering
\begin{tikzpicture}[scale=0.6]
\def\a{0.9};
\def\b{0.4};
\def\h{3.0};
\def\r{0.3};
\newcommand\Rec[3]{
\node at (#1,#2) {#3};
\draw(#1-\a,#2-\b+\r)--(#1-\a,#2+\b-\r);
\draw(#1-\a+\r,#2-\b)--(#1+\a-\r,#2-\b);
\draw(#1+\a,#2-\b+\r)--(#1+\a,#2+\b-\r);
\draw(#1-\a+\r,#2+\b)--(#1+\a-\r,#2+\b);
\draw (#1-\a,#2-\b+\r) arc (180:270:\r);
\draw (#1+\a-\r,#2-\b) arc (270:360:\r);
\draw(#1+\a,#2+\b-\r) arc (0:90:\r);
\draw(#1-\a+\r,#2+\b) arc (90:180:\r);
}
\Rec{0}{0}{$1234$}
\Rec{0}{\h}{$1324$}
\Rec{-4*\a}{\h}{$1243$}
\Rec{4*\a}{\h}{$2134$}
\Rec{-8*\a}{2*\h}{$1342$}
\Rec{-4*\a}{2*\h}{$1423$}
\Rec{0*\a}{2*\h}{$2143$}
\Rec{4*\a}{2*\h}{$2314$}
\Rec{8*\a}{2*\h}{$3124$}
\Rec{-6*\a}{3*\h}{$1432$}
\Rec{-2*\a}{3*\h}{$2413$}
\Rec{2*\a}{3*\h}{$3142$}
\Rec{6*\a}{3*\h}{$3214$}
\Rec{0*\a}{4*\h}{$3412$}

\draw(0,\b)--(-4*\a,\h-\b);
\draw(0,\b)--(0*\a,\h-\b);
\draw(0,\b)--(4*\a,\h-\b);
\draw(-4*\a,\h+\b)--(-8*\a,2*\h-\b);
\draw(-4*\a,\h+\b)--(-4*\a,2*\h-\b);
\draw(-4*\a,\h+\b)--(0*\a,2*\h-\b);
\draw(0*\a,\h+\b)--(-8*\a,2*\h-\b);
\draw(0*\a,\h+\b)--(-4*\a,2*\h-\b);
\draw(0*\a,\h+\b)--(4*\a,2*\h-\b);
\draw(0*\a,\h+\b)--(8*\a,2*\h-\b);
\draw(4*\a,\h+\b)--(0*\a,2*\h-\b);
\draw(4*\a,\h+\b)--(4*\a,2*\h-\b);
\draw(4*\a,\h+\b)--(8*\a,2*\h-\b);
\draw(-8*\a,2*\h+\b)--(-6*\a,3*\h-\b);
\draw[very thick](-8*\a,2*\h+\b)--(2*\a,3*\h-\b);
\draw(-4*\a,2*\h+\b)--(-6*\a,3*\h-\b);
\draw(-4*\a,2*\h+\b)--(-2*\a,3*\h-\b);
\draw(0*\a,2*\h+\b)--(-2*\a,3*\h-\b);
\draw[very thick](0*\a,2*\h+\b)--(2*\a,3*\h-\b);
\draw(4*\a,2*\h+\b)--(-2*\a,3*\h-\b);
\draw(4*\a,2*\h+\b)--(6*\a,3*\h-\b);
\draw[very thick](8*\a,2*\h+\b)--(2*\a,3*\h-\b);
\draw(8*\a,2*\h+\b)--(6*\a,3*\h-\b);
\draw[very thick](-6*\a,3*\h+\b)--(0*\a,4*\h-\b);
\draw[very thick](-2*\a,3*\h+\b)--(0*\a,4*\h-\b);
\draw[very thick](2*\a,3*\h+\b)--(0*\a,4*\h-\b);
\draw[very thick](6*\a,3*\h+\b)--(0*\a,4*\h-\b);
\end{tikzpicture}
\caption{The lower interval in the Bruhat order below $3412$}
\label{fig:3412-lower-interval}
\end{figure}
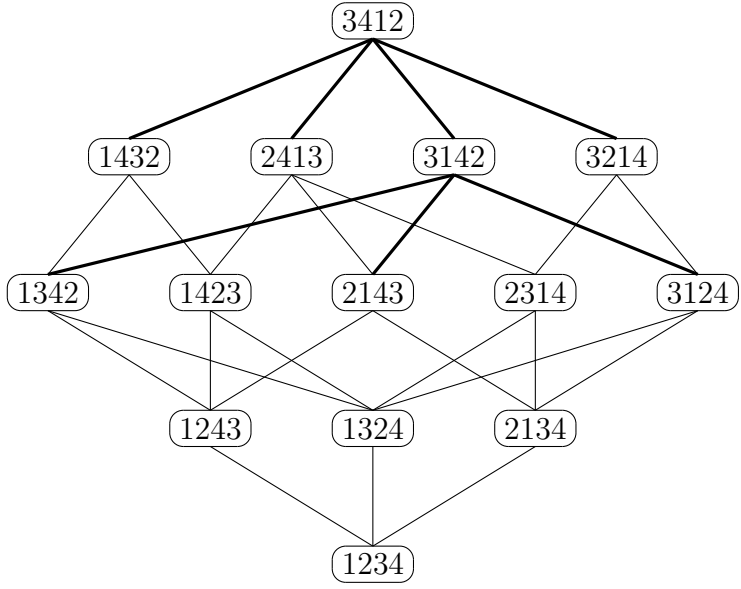
\end{ex}

\begin{proposition}\label{prop:AD-usual-recurrence}
$\AD(u,v)$ has the following properties.
\begin{enumerate}
\item For any $u\leq w\lessdot v$, $\AD(u,v)=\AD(u,w)+\mathbb{R}\cdot\wt(w,v)$.
\item For any saturated chain $u=w^{(0)}\lessdot w^{(1)}\lessdot\cdots\lessdot w^{(\ell-1)}\lessdot w^{(\ell)}=v$, \[\AD(u,v)=\mathrm{span}_{\mathbb{R}}\{\wt(w^{(i)},w^{(i+1)})\:|\: i=0,\ldots,\ell-1\}.\]
\item For any $w\in[u,v]$, $\AD(u,v)$ is spanned by the labels of all cover relations incident to $w$ inside $[u,v]$.
\end{enumerate}
\end{proposition}
\begin{proof}
We apply two layers of induction: first on $\ell=\ell(v)-\ell(u)$ and second on $\ell(u)$. The base cases where $\ell=0,1$ are true by definition. Assume $\ell\geq2$. Denote these statements as ($1_{\ell}$), ($2_{\ell}$) and ($3_{\ell}$) respectively.

($1_{\ell}$)$\Longrightarrow$($2_{\ell}$). For any such saturated chain, 
\begin{align*}
&\mathrm{span}_{\mathbb{R}}\{\wt(w^{(i)},w^{(i+1)})\:|\: 0\leq i\leq \ell-1\}\\
=&\mathrm{span}_{\mathbb{R}}\{\wt(w^{(i)},w^{(i+1)})\:|\: 0\leq i\leq \ell-2\}+\mathbb{R}\cdot\wt(w^{(\ell-1)},v)\\
=&\AD(u,w^{(\ell-1)})+\mathbb{R}\cdot\wt(w^{(\ell-1)},v)=\AD(u,v)
\end{align*}
by ($2_{\ell-1}$) and ($1_{\ell}$).

($2_{\ell}$)$\Longrightarrow$($3_{\ell}$). By Lemma~\ref{lem:AD-spanned-by-covers}, the labels of the upper covers from $w$ inside $[u,v]$ span $\AD(w,v)$, which by (2), equals the span of a saturated chain $C$ from $w$ to $v$. Similarly, the labels of the lower covers from $w$ have the same span as a saturated chain $C'$ from $u$ to $w$. Putting the two chains together and by ($2_{\ell}$), we obtain the desired equality.

Thus, we now focus on proving ($1_{\ell}$). Fix $w=s_{\alpha}v$. If $v$ has a right descent $\alpha_i$ such that $vs_i\neq w$, go to Case 1 and Case 2. Otherwise, $v$ must have a unique right descent $\alpha_i$ satisfying $vs_i=w=s_{\alpha}v$, and we go to Case 3.

\textbf{Case 1:} $vs_i\neq w$ and $us_i<u$. Note that we don't require $vs_i\geq u$ here. By Lemma~\ref{lem:strong-weak-diamond}, both $vs_i$ and $s_{\alpha}v$ cover $s_{\alpha}vs_i$. Also, since $s_{\alpha}v>u$ and both have right descents at $\alpha_i$, $s_{\alpha}vs_i>us_i$. By Proposition~\ref{prop:AD-recursion} and induction hypothesis on $\ell(us_i)<\ell(u)$, 
\[\AD(u,v)=\AD(us_i,vs_i)=\AD(us_i,s_{\alpha}vs_i)+\mathbb{R}\alpha=\AD(u,s_{\alpha}v)+\mathbb{R}\alpha.\]

\textbf{Case 2:} $vs_i\neq w$ and $us_i>u$. In this case, by the lifting property (Proposition~\ref{prop:lifting}) on $u<v$, we have $u\leq vs_i$. Similarly as above, $s_{\alpha}v$ has a right descent at $\alpha_i$. By the lifting property on $s_{\alpha}v>u$, we have $s_{\alpha}vs_i\geq u$ and $s_{\alpha}v\geq us_i$. Let $\beta=\wt(u,us_i)$. By induction hypothesis on $\ell(vs_i)-\ell(u)<\ell(v)-\ell(u)$ and Proposition~\ref{prop:AD-recursion}, 
\[\AD(u,v)=\AD(u,vs_i)+\mathbb{R}\beta=\AD(u,s_{\alpha}vs_i)+\mathbb{R}\alpha+\mathbb{R}\beta=\AD(u,s_{\alpha}v)+\mathbb{R}\alpha.\]

\textbf{Case 3:} $vs_i=s_{\alpha}v$. We are in this case when $\alpha_i$ is the only right descent of $v$. Choose any $w\gtrdot z\geq u$ and let $w'\neq w$ be the unique element such that $v\gtrdot w'\gtrdot z$. Let $\beta=\wt(w',v)$, $\gamma=\wt(z,w)$ and $\tau=\wt(z,w')$. By Lemma~\ref{lem:Dyer-rank2}, $\mathrm{span}_{\mathbb{R}}\{\alpha,\gamma\}=\mathrm{span}_{\mathbb{R}}\{\beta,\tau\}$. Since $\alpha_i$ is the only right descent of $v$, Case 1 and Case 2 can be applied to $u\leq w'\lessdot v$, giving us
\begin{align*}
\AD(u,v)=&\AD(u,w')+\mathbb{R}\beta\\
=&\AD(u,z)+\mathbb{R}\tau+\mathbb{R}\beta\\
=&\AD(u,z)+\mathbb{R}\gamma+\mathbb{R}\alpha\\
=&\AD(u,w)+\mathbb{R}\alpha
\end{align*}
by induction hypothesis ($1_{\ell-1}$).
\end{proof}

\begin{definition}\label{def:toric-interval}
A Bruhat interval $[u,v]$ is \emph{toric} if $\ad(u,v)=\ell(v)-\ell(u)$.
\end{definition}
In light of Proposition~\ref{prop:AD-usual-recurrence}(2), we always have $\ad(u,v)\leq \ell(v)-\ell(u).$
\begin{remark}
Unfortunately, being a toric interval is not a combinatorial invariant: the Bruhat interval $[1324,4231]$ (an example given in Remark 4.15 \cite{mahir-toric}) is isomorphic to the boolean lattice $B_4$, but is not toric; in the meantime, there exist many Bruhat intervals that are isomorphic to $B_4$ and are toric (for example $[12345,51234]$). When $G$ is simple and simply connected and $\ell(v)-\ell(u)\leq r$, Can and Saha show that $[u,v]$ is toric if and only if it is a lattice~\cite{mahir-toric}. 
\end{remark}
In fact, knowing which intervals are toric provides us with full control over $\ad(u,v)$.
\begin{proposition}\label{prop:toric-ad}
$\ad(u,v)=\max_{w\in[u,v]}\ell(w)-\ell(u)$ where $[u,w]$ is toric.
\end{proposition}
The dual statement of Proposition~\ref{prop:toric-ad}: \[\ad(u,v)=\max_{w\in[u,v]}\ell(v)-\ell(w)\text{ where }[w,v]\text{ is toric},\] follows from the exact same argument. It is also clear from Proposition~\ref{prop:toric-ad} that the following more symmetric-looking statement holds: \[\ad(u,v)=\max_{u\leq x\leq y\leq v}\ell(y)-\ell(x)\text{ where }[x,y]\text{ is toric}.\]
\begin{proof}
For any $w\in[u,v]$ such that $[u,w]$ is toric, $\ell(w)-\ell(u)=\ad(u,w)\leq\ad(u,v)$. To show that $\ad(u,v)\leq\max_{w\in[u,v]}\ell(w)-\ell(u)$ where $[u,w]$ is toric, we use induction on $\ell(v)-\ell(u)$. Note that when $[u,v]$ is toric, the statement holds, and that by Lemma~\ref{lem:Dyer-rank2}, whenever $\ell(v)-\ell(u)\leq2$, $[u,v]$ is toric. Now assume that $[u,v]$ is not toric. It suffices to show that there exists $v\gtrdot v'\geq u$ such that $\ad(u,v)=\ad(u,v')$, since in which case, by induction hypothesis, \[\ad(u,v)=\ad(u,v')=\max_{\substack{w\in[u,v']\\ [u,w]\text{ is toric}}}\ell(w)-\ell(u)\leq\max_{\substack{w\in[u,v]\\ [u,w]\text{ is toric}}}\ell(w)-\ell(u)\leq\ad(u,v),\] as desired. We first establish a helpful result.
\begin{claim}\label{claim:swap-0-1}
Consider $x,y_1,y_2\in[u,v]$ with $y_1\ne y_2$ such that $v\gtrdot y\gtrdot x$ for $y\in\{y_1,y_2\}$. If $\ad(u,v)-\ad(u,x)=1$, then there exists $y\in\{y_1,y_2\}$ such that $\AD(u,v)=\AD(u,y)$.
\end{claim}
\begin{proof}[Proof of claim]
Let the labels $\wt(y_1,v)$, $\wt(y_2,v)$, $\wt(x,y_1)$, $\wt(x,y_2)$ be $\alpha,\beta,\gamma,\tau$, respectively. Assume the opposite that $\AD(u,y)\subsetneq \AD(u,v)$ for $y\in\{y_1,y_2\}$. Since $\AD(u,x)\subset\AD(u,y)$ and $\ad(u,v)-\ad(u,x)=1$, we must have that $\AD(u,x)=\AD(u,y)$ for $y\in\{y_1,y_2\}$, meaning that $\gamma,\tau\in\AD(u,x)$. By Lemma~\ref{lem:Dyer-rank2}, $\mathrm{span}_{\mathbb{R}}\{\gamma,\tau\}=\mathrm{span}_{\mathbb{R}}\{\alpha,\beta\}$, so $\alpha,\beta\in\AD(u,x)$. By Proposition~\ref{prop:AD-usual-recurrence}(1), $\AD(u,v)=\AD(u,y_1)+\mathbb{R}\alpha=\AD(u,x)$, contradicting $\ad(u,v)-\ad(u,x)=1$.
\end{proof}
Now, fix a chain $v=v_0\gtrdot v_1\gtrdot\cdots\gtrdot v_k=u$, and consider the sequence of integers $d_i=\ad(u,v_{i-1})-\ad(u,v_i)\in\{0,1\}$ for $i=1,\ldots,k$. By Proposition~\ref{prop:AD-usual-recurrence}(2), $\sum_{i=1}^k d_i=\ad(u,v)<k$ because $[u,v]$ is not toric. Thus, there must be some zeros among the $d_i's$. We are going to tweak this chain to move $0$'s upwards.

If $d_j=1$ and $d_{j+1}=0$, let $v_j'\neq v_j$ be the unique element in $[v_{j+1},v_{j-1}]$. As $\ad(u,v_{j-1})-\ad(u,v_{j+1})=1$, by Claim~\ref{claim:swap-0-1}, there exists $y\in\{v_j,v_{j}'\}$ such that $\AD(u,y)=\AD(u,v_{j-1})$. Since $d_j=1$, $y\neq v_j$ and thus $y=v_j'$. Replace $v_j$ by $v_j'$ in the original chain, we see that $d_j'=0$ and $d_{j+1}=1$. Continue this process of moving $0$'s upwards, we will eventually end up with $v_1'\lessdot v$ such that $d_1'=0$, meaning that $\ad(u,v)=\ad(u,v_1')$ as desired.
\end{proof}

\begin{corollary}\label{cor:toric-interval-richardson}
For $u\leq v$, the Bruhat interval $[u,v]$ is toric if and only if the Richardson variety $\mathcal{R}_{u,v}$ is toric.
\end{corollary}

\begin{proof}
By Definition~\ref{def:toric-interval}, the interval $[u,v]$ is toric if and only if
\[
\ad(u,v)=\ell(v)-\ell(u).
\]
By Proposition~\ref{prop:AD-BIP}, this is equivalent to
\[
\dim Q_{u,v}=\ell(v)-\ell(u).
\]
By \cite[Proposition~7.12]{TW15}, this is equivalent to $\mathcal{R}_{u,v}$ being toric.
\end{proof}

\begin{proof}[Proof of Theorem~\ref{thm:torus-complexity}]
By Corollary~\ref{cor:max-t-orbit-ad}, the maximum dimension of a $T$-orbit in $\mathcal{R}_{u,v}$ is $\ad(u,v)$. Since $\dim\mathcal{R}_{u,v}=\ell(v)-\ell(u)$, Definition~\ref{def:complexity} gives
\[
c_T(\mathcal{R}_{u,v})=\ell(v)-\ell(u)-\ad(u,v).
\]

For the second assertion, use the dual form of Proposition~\ref{prop:toric-ad}:
\[
\ad(u,v)=\max_{w\in[u,v]}\ell(v)-\ell(w)\text{ where }[w,v]\text{ is toric}.
\]
By Corollary~\ref{cor:toric-interval-richardson}, the condition that $[w,v]$ is toric is equivalent to the condition that $\mathcal{R}_{w,v}$ is toric. Therefore
\[
\ad(u,v)=\max_{u\leq w\leq v}\ell(v)-\ell(w)\text{ where }\mathcal{R}_{w,v}\text{ is toric},
\]
as desired.
\end{proof}

\begin{proof}[Proof of Corollary~\ref{cor:torus-complexity-Schubert}]
The subword property (Proposition~\ref{prop:subword-property}) and Lemma~\ref{lem:AD-spanned-by-covers} imply that $\ad(\mathrm{id},w)=\supp(w)$. Since $\mathcal{R}_{\mathrm{id},w}=X_w$, the statement now follows from Theorem~\ref{thm:torus-complexity}.
\end{proof}

\begin{remark}
The Coxeter-theoretic arguments above, through Proposition~\ref{prop:toric-ad}, can be generalized to any Coxeter group with a chosen Coxeter datum. We choose to stay within the realm of finite Weyl groups for consistency of notation throughout this paper.
\end{remark}

\subsection{Torus weights via Deodhar decompositions}\label{sub:torus-complexity}
The preceding proof gives the global maximum dimension of a $T$-orbit in $\mathcal{R}_{u,v}$. We now refine this calculation by studying the torus action on individual Deodhar components. We will follow the notation of \cite{marsh-rietsch}.
\begin{definition}\cite[Section~3]{marsh-rietsch}\label{def:subexpressions}
For an expression $\textbf{w}=q_{i_1}\cdots q_{i_{\ell}}$ where each $q_{i_k}\in S\cup\{1\}$, define $w_{(k)}=q_{i_1}\cdots q_{i_k}$ for $k=1,\ldots,\ell$. Set $w_{(0)}=1$. Define the following sets:
\begin{align*}
J_{\mathbf{w}}^{+}:=&\{k\in\{1,\ldots,\ell\}\:|\: w_{(k-1)}<w_{(k)}\},\\
J_{\mathbf{w}}^{\circ}:=&\{k\in\{1,\ldots,\ell\}\:|\: w_{(k-1)}=w_{(k)}\},\\
J_{\mathbf{w}}^{-}:=&\{k\in\{1,\ldots,\ell\}\:|\: w_{(k-1)}>w_{(k)}\}.
\end{align*}
For a reduced expression $\textbf{v}=s_{i_1}\cdots s_{i_{\ell}}$, a \emph{subexpression} is obtained from $\textbf{v}$ by replacing some of the factors with $1$. A subexpression $\mathbf{u}$ of $\mathbf{v}$ is called \emph{distinguished} if $u_{(j)}\leq u_{(j-1)}s_{i_j}$ for all $j=1,\ldots,\ell$. In other words, if multiplication by $s_{i_j}$ (on the right) decreases the length of $u_{(j-1)}$, then we must have $u_{(j)}=u_{(j-1)}s_{i_j}$ in a distinguished subexpression. 
\end{definition}

For every simple root $\alpha_i\in\Delta$, there is a corresponding homomorphism $\varphi_i:\SL_2\rightarrow G$. Define the following elements \[x_{\alpha_i}(m)=\varphi_i\begin{pmatrix}1&m\\0&1\end{pmatrix},\quad y_{\alpha_i}(m)=\varphi_i\begin{pmatrix}1&0\\m&1\end{pmatrix}, \quad \dot s_i=\varphi_i\begin{pmatrix}0&-1\\1&0\end{pmatrix}.\]
For the majority of this work, we abuse notation and refer to the coset representative of $w \in W =N_G(T)/T$ in $N_G(T)$ by $w$. In their work, Marsh and Rietsch fix specific coset representatives. To stay in line with their notation and constructions, in this self-contained section,  we follow their notational conventions. For $s_i \in W$, we fix $\dot s_i$ as our coset representative of $s_i \in W$ in $N_G(T)$. For any $w \in W$ with reduced expression $w = s_{i_1} \cdots s_{i_\ell}$, let $\dot w := \dot s_{i_1} \cdots \dot s_{i_\ell}$ be  our coset representative of $w$ in $N_G(T)$. 

As introduced in Section~\ref{sub:root-subgroups}, each root $\alpha\in\Phi$ has an associated root subgroup $U_{\alpha}$. Choose a parameterization $x_{\alpha}:\mathbb{C}\to U_{\alpha}$ for each root, retaining the maps already defined for simple roots and taking $x_{-\alpha_i}=y_{\alpha_i}$. If $\alpha=w(\alpha_i)$, conjugation by $\dot w$ identifies $U_{\alpha_i}$ with $U_{\alpha}$, but the resulting parameterization can depend on the choice of $w$ and $i$. For every such parameterization, $t\in T$, and $m\in\mathbb{C}$,
\begin{equation} \label{eq:tactionrootsubgroup}
t x_{\alpha}(m) = x_{\alpha}(\alpha(t) m) t.
\end{equation}
\begin{definition}\cite[Definition~5.1]{marsh-rietsch}\label{def:deodhar-component}
For a reduced word $\mathbf{v}=s_{i_1}\cdots s_{i_{\ell}}$ and a distinguished subword $\mathbf{u}$, define the following set \[G_{\mathbf{u},\mathbf{v}}=\left\{g= g_1 g_2\cdots g_\ell \left
|\begin{array}{ll}
 g_k= x_{\alpha_{i_k}}(m_k) {\dot s_{i_k}}^{-1}& \text{ if $k\in J_{\mathbf{u}}^{-}$,}\\
 g_k= y_{\alpha_{i_k}}(p_k)& \text{ if $k\in J^{\circ}_{\mathbf{u}}$,}\\
 g_k=\dot s_{i_k}& \text{ if $k\in J^+_{\mathbf{u}}$,}
 \end{array}\text{    for $p_k\in\mathbb{C}^*,\, m_k\in\mathbb{C}$}\right. \right\}.\]
 We have $G_{\mathbf{u},\mathbf{v}}\simeq (\mathbb{C}^*)^{J_{\mathbf{u}}^{\circ}}\times \mathbb{C}^{J_{\mathbf{u}}^-}$. Now, the \emph{Deodhar component} is $D_{\mathbf{u},\mathbf{v}}:=G_{\mathbf{u},\mathbf{v}}B/B\subset G/B$. We have the isomorphism $G_{\mathbf{u},\mathbf{v}}\simeq D_{\mathbf{u},\mathbf{v}}$ via $g\mapsto gB/B$ (\cite[Proposition 5.2]{marsh-rietsch}).
\end{definition}
A distinguished subword $\mathbf{u}$ is called \emph{positive} if $u_{(j-1)}\leq u_{(j)}$ for all $j$, i.e. $J_{\mathbf{u}}^{-}=\emptyset$. Given a fixed reduced word $\mathbf{v}$ of $v$ and an element $u\leq v$, there is a unique positive distinguished subword $\mathbf{u}_+$ of $u$ by \cite[Lemma~3.5]{marsh-rietsch}. Since $J_{\mathbf{u}_+}^{-}=\emptyset$, the parameterization of $D_{\mathbf{u}_+,\mathbf{v}}$ gives $\dim D_{\mathbf{u}_+,\mathbf{v}}=|J_{\mathbf{u}_+}^{\circ}|=\ell(v)-\ell(u)=\dim\mathcal{R}_{u,v}$.

\begin{theorem}\cite[Theorem~1.1]{deodhar}
Given $u\leq v$ and fix a reduced word $\mathbf{v}$ of $v$. The Richardson cell has a decomposition $\R_{u,v}^{\circ}=\bigsqcup_{\mathbf{u}}D_{\mathbf{u},\mathbf{v}}$ into disjoint union, where the sum is over all distinguished subwords $\mathbf{u}$ of $\mathbf{v}$ that evaluate to $u$.
\end{theorem}

\begin{ex}
Let $v=321$ and $u=123$ in $S_3$. Pick a reduced word $\mathbf{v}=s_1s_2s_1$, and we have two distinguished subwords $s_11s_1$ and $111$ for $u$. Among these two distinguished subwords, $111$ is positive. The components look like
{ \renewcommand*{\arraystretch}{.8} \setlength{\arraycolsep}{3pt}
\begin{align*} G_{s_11s_1,s_1s_2s_1} & =\left\{\begin{pmatrix}&{-}1&\\1&&\\&&1\end{pmatrix}\begin{pmatrix}1&&\\&1&\\&p_2&1\end{pmatrix}\begin{pmatrix}{-}m_3&1&\\{-}1&&\\&&1\end{pmatrix}=\begin{pmatrix}1&&\\{-}m_3&1&\\{-}p_2&&1\end{pmatrix}\middle|\ p_2\in\mathbb{C}^*,m_3\in\mathbb{C}\right\},  \\ 
G_{111,s_1s_2s_1} & =\left\{\begin{pmatrix}1&&\\p_1&1&\\&&1\end{pmatrix}\begin{pmatrix}1&&\\&1&\\&p_2&1\end{pmatrix}\begin{pmatrix}1&&\\p_3&1&\\&&1\end{pmatrix}=\begin{pmatrix}1&&\\p_1{+}p_3&1&\\p_2p_3&p_2&1\end{pmatrix}\middle|\ p_1,p_2,p_3\in\mathbb{C}^*\right\} .\end{align*}}
Then we look at how the torus acts on specific points in the Deodhar components.
{ \renewcommand*{\arraystretch}{.8} \setlength{\arraycolsep}{3pt}
\begin{align*}
&\color{red}\begin{pmatrix}t_1&&\\&t_2&\\&&t_3\end{pmatrix}\color{black}\begin{pmatrix}&{-}1&\\1&&\\&&1\end{pmatrix}\begin{pmatrix}1&&\\&1&\\&p_2&1\end{pmatrix}\begin{pmatrix}{-}m_3&1&\\{-}1&&\\&&1\end{pmatrix}B/B\\
=&\begin{pmatrix}&{-}1&\\1&&\\&&1\end{pmatrix}\color{red}\begin{pmatrix}t_2&&\\&t_1&\\&&t_3\end{pmatrix}\color{black}\begin{pmatrix}1&&\\&1&\\&p_2&1\end{pmatrix}\begin{pmatrix}{-}m_3&1&\\{-}1&&\\&&1\end{pmatrix}B/B\\
=&\begin{pmatrix}&{-}1&\\1&&\\&&1\end{pmatrix}\begin{pmatrix}1&&\\&1&\\&p_2\frac{t_3}{t_1}&1\end{pmatrix}\color{red}\begin{pmatrix}t_2&&\\&t_1&\\&&t_3\end{pmatrix}\color{black}\begin{pmatrix}{-}m_3&1&\\{-}1&&\\&&1\end{pmatrix}B/B\\
=&\begin{pmatrix}&{-}1&\\1&&\\&&1\end{pmatrix}\begin{pmatrix}1&&\\&1&\\&p_2\frac{t_3}{t_1}&1\end{pmatrix}\begin{pmatrix}{-}m_3\frac{t_2}{t_1}&1&\\{-}1&&\\&&1\end{pmatrix}\color{red}\begin{pmatrix}t_1&&\\&t_2&\\&&t_3\end{pmatrix}\color{black}B/B\\
=&\begin{pmatrix}&{-}1&\\1&&\\&&1\end{pmatrix}\begin{pmatrix}1&&\\&1&\\&p_2\frac{t_3}{t_1}&1\end{pmatrix}\begin{pmatrix}{-}m_3\frac{t_2}{t_1}&1&\\{-}1&&\\&&1\end{pmatrix}B/B,
\end{align*}}where the torus ``disappears" from the right since $T\subset B$. Thus, we can say that $T$ acts on $D_{s_11s_1,s_1s_2s_1}$ by weights $e_3-e_1$ and $e_2-e_1$, which span a $2$-dimensional space in the root space. With a similar calculation,
{ \renewcommand*{\arraystretch}{.8} \setlength{\arraycolsep}{3pt}
\begin{align*}
&\color{red}\begin{pmatrix}t_1&&\\&t_2&\\&&t_3\end{pmatrix}\color{black}\begin{pmatrix}1&&\\p_1&1&\\&&1\end{pmatrix}\begin{pmatrix}1&&\\&1&\\&p_2&1\end{pmatrix}\begin{pmatrix}1&&\\p_3&1&\\&&1\end{pmatrix}B/B\\
=&\begin{pmatrix}1&&\\p_1\frac{t_2}{t_1}&1&\\&&1\end{pmatrix}\begin{pmatrix}1&&\\&1&\\&p_2\frac{t_3}{t_2}&1\end{pmatrix}\begin{pmatrix}1&&\\p_3\frac{t_2}{t_1}&1&\\&&1\end{pmatrix}B/B,
\end{align*}}so $T$ acts on $D_{111,s_1s_2s_1}$ by weights $e_2-e_1,e_3-e_2,e_2-e_1$, whose span is $2$-dimensional. 
\end{ex}

\begin{definition}\label{def:root-for-subword}
Fix a reduced word $\mathbf{v}$ of $v$ and a distinguished subword $\mathbf{u}$. For each $k\in J_{\mathbf{u}}^{\circ}$, define $\beta_k=\wt(u_{(k-1)},u_{(k-1)}s_{i_k})$; for each $k\in J_{\mathbf{u}}^-$, define $\beta_k=\wt(u_{(k)},u_{(k-1)})$ (cf. Definition~\ref{def:Bruhat-graph}). Note that in both cases, $s_{\beta_k}=u_{(k-1)}s_{i_k}u_{(k-1)}^{-1}$ so $\beta_k=\pm u_{(k-1)}(\alpha_{i_k})$. To be precise, when $k\in J_{\mathbf{u}}^{\circ}$, $u_{(k-1)}<u_{(k-1)}s_{i_k}$ and when $k\in J_{\mathbf{u}}^-$, $u_{(k-1)}>u_{(k-1)}s_{i_k}$, so we have that $\beta_k=u_{(k-1)}(\alpha_{i_k})$ when $k\in J_{\mathbf{u}}^{\circ}$ and $\beta_k=-u_{(k-1)}(\alpha_{i_k})$ when $k\in J_{\mathbf{u}}^-$. Also define \[U_k = \begin{cases} U^{*}_{-\beta_k} & \textrm{if }k \in J_{\mathbf{u}}^{\circ} \\ U_{-\beta_k} & \textrm{if }k \in J_{\mathbf{u}}^- \\ \end{cases}.\] 

For $k\in J_{\mathbf{u}}^{\circ}\cup J_{\mathbf{u}}^-$, define the parameterization $x_k:\mathbb{C}\to U_{-\beta_k}$ by
\[
x_k(m)=\begin{cases}
\dot u_{(k-1)}y_{\alpha_{i_k}}(m)\dot u_{(k-1)}^{-1} & \text{if }k\in J_{\mathbf{u}}^{\circ},\\
\dot u_{(k-1)}x_{\alpha_{i_k}}(m)\dot u_{(k-1)}^{-1} & \text{if }k\in J_{\mathbf{u}}^-.
\end{cases}
\]
The index $k$ records the parameterization as well as the root, so these maps need not agree when two positions have the same root $\beta_k$.

Define $\TD(\mathbf{u},\mathbf{v})=\mathrm{span}_{\mathbb{R}}\{\beta_k\:|\: k\in J_{\mathbf{u}}^{\circ}\cup J_{\mathbf{u}}^-\}$ and $\td(\mathbf{u},\mathbf{v})=\dim\TD(\mathbf{u},\mathbf{v})$. 

For any $x\in D_{\mathbf{u},\mathbf{v}}$, there is a unique $g\in G_{\mathbf{u},\mathbf{v}}$ such that the coset $gB$ equals $x$. In this case, call $g$ the \emph{standard form} of $x$. For a $g \in G_{\mathbf{u},\mathbf{v}}$ and $k\in J_{\mathbf{u}}^{\circ}\cup J_{\mathbf{u}}^-$, let
$$g^{(k)} = \begin{cases} p_k & \textrm{if }k \in J_{\mathbf{u}}^{\circ} \\ m_k & \textrm{if }k \in J_{\mathbf{u}}^- \\ \end{cases}.$$
We say that $x$ is a \emph{general} point in $D_{\mathbf{u},\mathbf{v}}$ if $g^{(k)}$ is nonzero for all $k \in J_{\mathbf{u}}^-$.
\end{definition}

\begin{lemma}
\label{lemma:DeodharcomponentUnipotentSubgroup}
Fix $u\leq v$ and a reduced word $\mathbf{v}=s_{i_1}\cdots s_{i_{\ell}}$ of $v$. Let $\mathbf{u}=u_1\cdots u_\ell$ be a distinguished subword of $\mathbf{v}$ for $u$. Let $J_{\mathbf{u}}^{\circ}\cup J_{\mathbf{u}}^- = \{ j_1 < j_2 < \cdots < j_z \}$ and $g \in G_{\mathbf{u},\mathbf{v}}$. Then $$g = \displaystyle \prod_{k=1}^z x_{j_k}(g^{(j_k)}) \dot u, \qquad D_{\mathbf{u},\mathbf{v}} = \displaystyle \prod_{k=1}^{z} U_{j_k} \dot u B.$$
\end{lemma}
\begin{proof}
The first equality directly implies the second, and hence we focus on proving the first via induction on $\ell$. When $\ell=1$, $u = u_1 = u_{(1)}$. If $u_1 = s_{i_1}$, then $J_{\mathbf{u}}^{+} = \{ 1 \}$ and $g = \dot s_{i_1}$. Otherwise, if $u_1 = 1$, then $J_{\mathbf{u}}^{\circ} = \{ 1 \}$ and $g = y_{\alpha_{i_1}}(g^{(1)}) = x_1(g^{(1)})$.

Assume $\ell \geq 2$. By the inductive hypothesis, applied to $\mathbf{v'}=s_{i_1}\cdots s_{i_{\ell-1}}$ and $\mathbf{u'} = u_1\cdots u_{\ell - 1}$ a distinguished subword of $\mathbf{v'}$ for $u_{(\ell-1)}$, 
\begin{equation}\label{eq:inductiveg}
g_1 \cdots g_{\ell - 1} = \displaystyle \prod_{k=1}^{z'} x_{j_k}(g^{(j_k)}) \dot u_{(\ell-1)},
\end{equation}
where $J_{\mathbf{u'}}^{\circ}\cup J_{\mathbf{u'}}^- = \{ j_1 < j_2 < \cdots < j_{z'} \} \subseteq J_{\mathbf{u}}^{\circ}\cup J_{\mathbf{u}}^-$ and $z'$ either equals $z$ or $z-1$.

There are three possible cases to consider.

\textbf{Case 1:} $\ell \in J_{\mathbf{u}}^{+}$. In this case, $\dot u_{(\ell - 1)} g_\ell = \dot u_{(\ell - 1)} \dot s_{i_\ell} = \dot u_{(\ell)} = \dot u$. In addition, $J_{\mathbf{u}}^{\circ} \cup J_{\mathbf{u}}^{-} = J_{\mathbf{u'}}^{\circ} \cup J_{\mathbf{u'}}^{-}$ and $z=z'$. The two preceding sentences, combined with \eqref{eq:inductiveg}, imply $$g=\displaystyle \prod_{k=1}^{z'} x_{j_k}(g^{(j_k)}) \dot u_{(\ell-1)}g_\ell = \prod_{k=1}^{z} x_{j_k}(g^{(j_k)}) \dot u.$$

\textbf{Case 2:} $\ell \in J_{\mathbf{u}}^{\circ}$. By the definition of $x_\ell$,
\[
\dot u_{(\ell-1)}g_\ell
=\dot u_{(\ell-1)}y_{\alpha_{i_\ell}}(g^{(\ell)})
=x_\ell(g^{(\ell)})\dot u_{(\ell-1)}
=x_\ell(g^{(\ell)})\dot u,
\]
where the last equality uses $u_{(\ell-1)}=u$. In addition, $J_{\mathbf{u}}^{\circ} \cup J_{\mathbf{u}}^{-} = J_{\mathbf{u'}}^{\circ} \cup J_{\mathbf{u'}}^{-} \cup \{ \ell \}$ and $z=z'+1$. Applying \eqref{eq:inductiveg} yields $$g=\displaystyle \prod_{k=1}^{z'} x_{j_k}(g^{(j_k)}) \dot u_{(\ell - 1)} g_\ell = \prod_{k=1}^{z'} x_{j_k}(g^{(j_k)}) x_\ell(g^{(\ell)}) \dot u = \prod_{k=1}^{z} x_{j_k}(g^{(j_k)}) \dot u.$$

\textbf{Case 3:} $\ell \in J_{\mathbf{u}}^{-}$. By the definition of $x_\ell$,
\begin{align*}
\dot u_{(\ell-1)}g_\ell
&=\dot u_{(\ell-1)}x_{\alpha_{i_\ell}}(g^{(\ell)})\dot s_{i_\ell}^{-1}\\
&=x_\ell(g^{(\ell)})\dot u_{(\ell-1)}\dot s_{i_\ell}^{-1}\\
&=x_\ell(g^{(\ell)})\dot u.
\end{align*}
The last equality follows from $u=u_{(\ell-1)}s_{i_\ell}<u_{(\ell-1)}$, which implies $\dot u=\dot u_{(\ell-1)}\dot s_{i_\ell}^{-1}$~\cite[Proof of Proposition~5.2]{marsh-rietsch}. In addition, $J_{\mathbf{u}}^{\circ} \cup J_{\mathbf{u}}^{-} = J_{\mathbf{u'}}^{\circ} \cup J_{\mathbf{u'}}^{-} \cup \{ \ell \}$ and $z=z'+1$. Finally, by \eqref{eq:inductiveg}
$$g=\displaystyle \prod_{k=1}^{z'} x_{j_k}(g^{(j_k)}) \dot u_{(\ell - 1)} g_\ell = \prod_{k=1}^{z'} x_{j_k}(g^{(j_k)}) x_\ell(g^{(\ell)}) \dot u =\prod_{k=1}^{z} x_{j_k}(g^{(j_k)}) \dot u.$$

As all three cases yield the desired formula for $g$ our proof is complete.
\end{proof}

\begin{proof}[Proof of Theorem~\ref{thm:torus-dimension-richardson}]
The first part of the proof utilizes similar arguments as Theorem 4.7 of \cite{mahir-toric}, who analyzed the torus weights for positive distinguished subwords. 

By Lemma~\ref{lemma:DeodharcomponentUnipotentSubgroup}, the $T$-action on $D_{\mathbf{u},\mathbf{v}}$ is given by the $T$-action on $\prod_{k=1}^{z} U_{j_k} \dot u B$, where $J_{\mathbf{u}}^{\circ}\cup J_{\mathbf{u}}^- = \{ j_1 < j_2 < \cdots < j_z \}$. For $h \in T$ and $g$ the standard form of $x \in D_{\mathbf{u},\mathbf{v}}$,
\[
h \cdot x = h \cdot \prod_{k=1}^{z} x_{j_k}(g^{(j_k)}) \dot u B
= \prod_{k=1}^{z} x_{j_k}\big((-\beta_{j_k})(h) g^{(j_k)}\big) \dot u B,
\]
where the final equality follows from \eqref{eq:tactionrootsubgroup} and the fact that $\dot uB$ is $T$-fixed. The Deodhar parameters are unique, so $h\cdot x=x$ if and only if $\beta_{j_k}(h)=1$ whenever $g^{(j_k)}\ne0$. Hence
\[
T_x=\bigcap_{\substack{1\leq k\leq z\\g^{(j_k)}\ne0}}\ker\beta_{j_k}.
\]
The codimension of this subgroup in $T$ is the dimension of the real span of the corresponding characters. Thus $\dim Tx=\dim T-\dim T_x$ gives the asserted orbit-dimension formula. For a general point, all the parameters $g^{(j_k)}$ are nonzero, so $\dim Tx=\td(\mathbf{u},\mathbf{v})$.

For the second part of the proof, we use induction on $\ell(v)$ to obtain a recurrence and relate it to recurrence of $\AD(u,v)$ (Proposition~\ref{prop:AD-recursion}). Let $\mathbf{v}=s_{i_1}\cdots s_{i_{\ell}}$ and write $s_i=s_{i_{\ell}}$ for simplicity, as we will mainly be discussing this right descent of $v$. Write $\mathbf{u}=q_{i_1}\cdots q_{i_{\ell}}$ where $q_{i_j}\in\{s_{i_j},1\}$. Let $\mathbf{v}'=s_{i_1}\cdots s_{i_{\ell-1}}$ and $\mathbf{u}'=q_{i_1}\cdots q_{i_{\ell-1}}$. Prefixes of a distinguished subword are also distinguished so $\mathbf{u}'$ is a distinguished subword of $\mathbf{v}'$. Let $\mathbf{v}'$ evaluate to $v'=vs_i<v$ and let $\mathbf{u}'$ evaluate to $u'\in\{us_i,u\}$.

\textbf{Case 1:} $us_i<u$. If $u'=u$, then by the definition of a distinguished subword (Definition~\ref{def:subexpressions}), we need to select $q_{i_{\ell}}=s_i$, contradicting $u=u'q_{i_{\ell}}$. Thus, $u'=us_i$, $q_{i_{\ell}}=s_i$ and $\ell\in J_{\mathbf{u}}^+$. The sequence of $\beta_k$'s (Definition~\ref{def:root-for-subword}) remains unchanged from $(\mathbf{u},\mathbf{v})$ to $(\mathbf{u}',\mathbf{v}')$ and $\mathbf{u}$ is positive if and only if $\mathbf{u}'$ is positive. In addition, $\AD(u,v)=\AD(us_i,vs_i)$ by Proposition~\ref{prop:AD-recursion}. Thus, the desired statement follows from the induction hypothesis $\TD(\mathbf{u}',\mathbf{v}')\subseteq\TD(\mathbf{u}'_+,\mathbf{v}')=\AD(us_i,vs_i)$.

\textbf{Case 2:} $us_i>u$. Here, either $q_{i_{\ell}}=1$ or $q_{i_{\ell}}=s_i$ is possible for a distinguished subword $\mathbf{u}$. The positive distinguished subword must have $q_{i_{\ell}}=1$. In both situations, $\beta_{\ell}=\wt(u,us_i)$ and $\TD(\mathbf{u},\mathbf{v})=\TD(\mathbf{u}',\mathbf{v}')+\mathbb{R}\beta_{\ell}$.

\textbf{Subcase 2.1:} $q_{i_{\ell}}=1$, so $\ell\in J_{\mathbf{u}}^{\circ}$. In this case, $\mathbf{u}$ is positive if and only if $\mathbf{u}'$ is positive. By the induction hypothesis and Proposition~\ref{prop:AD-recursion},
\[
\TD(\mathbf{u}_+,\mathbf{v})=\TD(\mathbf{u}'_+,\mathbf{v}')+\mathbb{R}\beta_{\ell}=\AD(u,vs_i)+\mathbb{R}\beta_{\ell}=\AD(u,v).
\]
Since $\TD(\mathbf{u}',\mathbf{v}')\subseteq\TD(\mathbf{u}'_+,\mathbf{v}')$ by induction, adding $\mathbb{R}\beta_{\ell}$ to both sides gives $\TD(\mathbf{u},\mathbf{v})\subseteq\TD(\mathbf{u}_+,\mathbf{v})$.

\textbf{Subcase 2.2:} $q_{i_{\ell}}=s_i$, so $\ell\in J_{\mathbf{u}}^{-}$. In this case, $\mathbf{u}$ is not positive and $\mathbf{u}'$ is a distinguished subword for $u'=us_i>u$. By induction, $\TD(\mathbf{u}',\mathbf{v}')\subseteq \AD(us_i,vs_i)$. Hence
\[
\TD(\mathbf{u},\mathbf{v})=\TD(\mathbf{u}',\mathbf{v}')+\mathbb{R}\beta_{\ell}\subseteq\AD(us_i,vs_i)+\mathbb{R}\beta_{\ell}\subseteq\AD(u,vs_i)+\mathbb{R}\beta_{\ell}=\AD(u,v).
\]
Subcase 2.1 gives $\TD(\mathbf{u}_+,\mathbf{v})=\AD(u,v)$, so $\TD(\mathbf{u},\mathbf{v})\subseteq\TD(\mathbf{u}_+,\mathbf{v})$. This completes the proof.
\end{proof}
Theorem~\ref{thm:torus-dimension-richardson} gives a Deodhar-stratum-wise refinement of Corollary~\ref{cor:max-t-orbit-ad}, and hence of the first assertion of Theorem~\ref{thm:torus-complexity}.

\section{The Levi complexity of Schubert varieties}\label{sec:levi-complexity}
\subsection{An equivariant isomorphism}
In \cite[Proposition~3.7]{GHY24}, the authors of this work and A.~Yong showed that if $w = w_0(I) d$ is length additive for an $I \subseteq \Delta$, then $X^{\circ}_w$ contains a distinguished subspace that is $T$-equivariantly isomorphic to $X^{\circ}_d$ equipped with the twisted action $\bigdot_{w_0(I)}$ defined below. The content in the first portion of this section is similar in broad form to our argument from \cite{GHY24}, but our results apply to any $w \in W$ and $I \subseteq \Delta$. When $L_I$ acts on $X_w$, this leads to a codimension preserving bijection between the $B_{L_I}$-orbits of $X_{w}^{\circ}$ and the $T$-orbits of $X_{\prescript{I}{}{w}}^{\circ}$.

Let $X$ be a $T$-variety with action denoted by $\bigdot$. For each $u \in W$ we define a $T$-action $\bigdot_u$ on $X$ by $t \bigdot_u x = utu^{-1} \bigdot x$ for all $x \in X$ and $t \in T$. The action $\bigdot_u$ does not depend on the coset representative chosen for $u$ and so is well-defined by \cite[Lemma 3.1]{GHY24}. Henceforth, $\bigdot$ will always denote the usual action of $T$ on $G/B$ by left multiplication. 

For a positive integer $n$ let $\mathbb{A}^n$ be the \emph{affine n-space}. We use the following well-known fact~\cite[\S 14.12, \S 14.4]{B91} repeatedly in what follows. For all $v \in W$, 
\begin{equation} 
\label{eq:SchubertcellAffine}
    X_{v}^{\circ} = U_v v B \cong \mathbb{A}^{\ell(v)} \text{(as varieties)}.
\end{equation}

From now on, to simplify our notation let the left parabolic decomposition of $w$ with respect to $I\subseteq\Delta$ be $w=ad$ where $a\in W_I$ and $d\in \prescript{I}{}{W}$. Recall the definition of $V_d$ in Subsection~\ref{sub:root-subgroups}.
\begin{definition} 
The \emph{$I$-heart} of the Schubert variety $X_w$, denoted $\Heart_I(X_w)$, is the subvariety $V_d w B \subseteq U_w w B = X_w^{\circ}$.
\end{definition}

\begin{proposition}
\label{proposition:tStableHeartIso}
$\Heart_I(X_w)$ is $T$-stable for the action $\bigdot$. Additionally, $\Heart_I(X_w)$ is $T$-equivariantly isomorphic to $X_d^{\circ}$ equipped with the $T$-action $\bigdot_{a^{-1}}$ under the map $\phi: \Heart_I(X_w) \rightarrow X_{d}^{\circ}$ given by $hB \mapsto a^{-1}hB$. 
\end{proposition}
\begin{proof}
Let $t \in T$. Then
\[
t \bigdot \Heart_I(X_w) = t V_d w B = t (t^{-1}V_d t) w B = V_d w B = \Heart_I(X_w),
\]
where the second equality follows from the fact that $V_d$ is normalized by $T$, and the third inequality from $wB$ being a $T$-fixed point.  

The image of the map $\phi$ is $\phi(\Heart_I(X_w))= a^{-1} V_d w B = U_d d B = X_d^{\circ}$, where the second equality is the definition of $V_d$ and the third equality is \eqref{eq:SchubertcellAffine}. Thus $\phi$ is well-defined and surjective. Its inverse is the restriction of left multiplication by $a$, so $\phi$ is an isomorphism of varieties. Now, let $t \in T$ and $hB \in \Heart_I(X_w)$. We have
\[
\phi(t \bigdot hB) = a^{-1}thB = a^{-1}ta a^{-1} hB = t \bigdot_{a^{-1}} \phi(hB),
\]
establishing the $T$-equivariance of $\phi$.
\end{proof}

Define $B_{L_I} := L_I \cap B$ and $U_{L_I}$ to be the unipotent radical of $B_{L_I}$. Then $B_{L_I} = T \ltimes U_{L_I}$ is a Borel subgroup of $L_I$, and $U_{L_I} = B_{L_I} \cap U = \prod_{\alpha \in \Phi^+(I)} U_{\alpha}$ \cite[\S 14]{B91}.

Set $U_{\bar{a}} = \prod_{\alpha \in \Phi^+(I) \setminus {\mathcal I}_L(a)} U_{\alpha}$. If $\Phi^+(I) \setminus {\mathcal I}_L(a) = \emptyset$, then we define $U_{\bar{a}}$ to be the trivial subgroup of $G$.  Then $U_{\bar{a}}$ is a closed subgroup of $U$ normalized by $T$ and
\[B_{L_I} = T \ltimes U_{L_I} =  T \ltimes \prod_{\alpha \in \Phi^+(I)} U_{\alpha} = T \ltimes (U_{\bar{a}} U_a).\]

\begin{lemma}
\label{lemma:equivalentCharUbara}
The followings are equivalent:
\begin{enumerate}
\item $U_{\bar{a}}$ is the trivial subgroup of $G$.
\item $\Phi^+(I) \setminus {\mathcal I}_L(a) = \emptyset$.
\item $L_I$ acts on $X_w$.
\item $I \subseteq \mathcal{D}_L(w)$.
\item $a = w_0(I)$.
\end{enumerate}
\end{lemma}
\begin{proof}
One direction of the equivalence between (1) and (2) is by construction and the other follows from the definition of root subgroups.

The equivalence of (3) and (4) is \cite[Lemma 8.2.3]{BL00}. If $I\subseteq {\mathcal D}_L(w)$, then ${\mathcal I}_L(w)$ contains all the positive roots in the root subsystem generated by $I$. Thus $w=w_0(I)d$ is a length-additive expression for some $d\in W$ by \cite[Proposition 3.1.3]{Bjorner.Brenti}. That is, $a = {_Iw} = w_0(I)$ in the left parabolic decomposition of $w$. For the other direction, if $a = w_0(I)$, then $I \subseteq \mathcal{D}_L(w)$. Hence, (4) and (5) are equivalent.

Finally, the equivalence of (2) and (5) is immediate.
\end{proof}

\begin{lemma}
\label{lemma:uBarStableHeart}
$\Heart_I(X_w)$ is $U_{\bar{a}}$-stable for the usual action.
\end{lemma}
\begin{proof}
We have that
\begin{align*}
    U_{\bar{a}} \Heart_I(X_w) &= U_{\bar{a}} a U_d a^{-1} w B = a a^{-1} U_{\bar{a}} a U_d d B = a \prod_{\alpha \in \Phi^+(I) \setminus {\mathcal I}_L(a)} U_{a^{-1}(\alpha)} U_d d B \\
                                                             &= a U_d d B = a U_d a^{-1} a d B = \Heart_I(X_w)
\end{align*}
where the third equality follows by \cite[Part II, 1.4(5)]{J03}, and the fourth equality is \eqref{eq:SchubertcellAffine} combined with the fact that $\alpha \notin {\mathcal I}_L(a)$ implies $U_{a^{-1}(\alpha)} \leq B$.
\end{proof}

\begin{lemma}
\label{lemma:orbitDynamicsRevisited}
Let $x \in X_{w}^{\circ} \setminus \Heart_I(X_w)$ and $h_1,h_2 \in \Heart_I(X_w)$.
\begin{enumerate}
\item $u_2h_1 \notin \Heart_I(X_w)$ for all $u_2 \in U_{a}$ with $u_2 \neq e$.
\item $tu_1x \notin \Heart_I(X_w)$ for all $t \in T$ and $u_1 \in U_{\bar{a}}$.
\item Let $b = t u_1 u_2 \in B_{L_I}$ with $t \in T$, $u_1 \in U_{\bar{a}}$, and $u_2 \in U_a$. If $bh_1=h_2$, then $u_2 = e$ and $t u_1 h_1 = h_2$.
\end{enumerate}
\end{lemma}
\begin{proof}
(1) Since $h_1 \in \Heart_I(X_w)$, $h_1 = v w B$ for some $v \in V_d$. By Lemma~\ref{lemma:directlyspannedUw}, $u_2v \notin V_d$. Thus $u_2 v w B \in  X_w^{\circ} \setminus \Heart_I(X_w)$ by \eqref{eq:SchubertcellAffine}.

(2) The $T$ and $U_{\bar{a}}$ stability of $\Heart_I(X_w)$, proved in Proposition ~\ref{proposition:tStableHeartIso} and Lemma~\ref{lemma:uBarStableHeart} respectively, imply $tu_1x \notin \Heart_I(X_w)$.

(3) Suppose, for contradiction that $u_2 \neq e$. Then (1) implies that $u_2 h_1 \notin \Heart_I(X_w)$. And thus, (2) implies $t u_1  u_2 h_1 \notin \Heart_I(X_w)$. This contradicts our assumption that $t u_1 u_2 h_1 = h_2$. Thus $u_2 = e$ and $t u_1 h_1 = b h_1 = h_2$.
\end{proof}

\begin{lemma}
\label{lemma:orbitStabilizer}
  Let $x \in X_w^{\circ}$. Then $(B_{L_I} \cdot x) \cap \Heart_I(X_w) \neq \emptyset$. 
\end{lemma}
\begin{proof}
Let $x \in X_w^{\circ}$. By \eqref{eq:SchubertcellAffine}, $x = uwB$ for $u \in U_w$. Thus, by Lemma~\ref{lemma:directlyspannedUw}, $x = u_1 v w B$ for $u_1 \in U_a, v \in V_d$. Since $u_1^{-1} \in U_a \leq B_{L_I}$, we have $v w B \in B_{L_I} \cdot x$ and $vwB \in \Heart_I(X_w)$ by definition. Thus $(B_{L_I} \cdot x) \cap \Heart_I(X_w) \neq \emptyset$.
\end{proof}

We now define our surjection from $T$-orbits in $\Heart_I(X_w)$ to $B_{L_I}$-orbits in $X_w^{\circ}$, which is a codimension preserving bijection in the case where $L_I$ acts on the Schubert variety .

\begin{theorem}
\label{theorem:codimensionPreservingOrbitBijection}
The map $\beta: \mathcal{O}_T(\Heart_I(X_w)) \rightarrow \mathcal{O}_{B_{L_I}}(X_w^{\circ})$ given by $\Theta\mapsto B_{L_I} x$
where $x$ is any point in $\Theta$ is a surjection. If $L_I$ acts on $X_w$, then $\beta$ is a codimension preserving bijection. That is, \[ \dim(X_w^{\circ}) - \dim(\beta(\Theta)) = \dim(\Heart_I(X_w)) - \dim(\Theta), \]
for every $\Theta \in \mathcal{O}_T(\Heart_I(X_w))$.
\end{theorem}
\begin{proof}
We begin by noting that $T \leq B_{L_I}$ implies that this map is well-defined. 

Let $\Xi \in \mathcal{O}_{B_{L_I}}(X_w^{\circ})$. By Lemma~\ref{lemma:orbitStabilizer}, there exists an $h \in \Xi \cap \Heart_I(X_w)$. Thus $h \in \Theta$ for some $\Theta \in \mathcal{O}_T(\Heart_I(X_w))$ and so $\beta(\Theta) = \Xi$. Thus $\beta$ is surjective.

For the remainder of this proof we assume that $L_I$ acts on $X_w$; that is, by Lemma~\ref{lemma:equivalentCharUbara},  $U_{\bar{a}}$ is the trivial subgroup of $G$. Let $\Theta_1, \Theta_2 \in \mathcal{O}_T(\Heart_I(X_w))$ with $x \in \Theta_1$ and $y \in \Theta_2$. Suppose that $B_{L_I} \cdot x = B_{L_I} \cdot y$. This implies $b_1 x = b_2 y$ for some $b_1,b_2 \in B_{L_I}$. Thus $b_2^{-1} b_1 x = y$. Lemma~\ref{lemma:orbitDynamicsRevisited}(3) implies there exists a $t \in T$ and $u \in U_{\bar{a}}$ such that $t u x = y$. Since $U_{\bar{a}}$ is the trivial subgroup, this implies $t x = y$. Thus $\Theta_1 = \Theta_2$ and $\beta$ is injective. Thus $\beta$ is a bijection.

Let $\Xi \in \mathcal{O}_{B_{L_I}}(X_w^{\circ})$. As above, there exists an $h \in \Xi \cap \Heart_I(X_w)$ by Lemma~\ref{lemma:orbitStabilizer}. Thus $h \in \Theta$ for a unique $\Theta \in \mathcal{O}_T(\Heart_I(X_w))$ since $\beta$ is a bijection. We claim that 
\[(B_{L_I})_h = T_h.\]
Clearly $T_h \subseteq (B_{L_I})_h$. For the other direction, suppose $b h = t u_1 u_2 h = h$ with $t \in T$, $u_1 \in U_{\bar{a}}$, and $u_2 \in U_a$. Then Lemma~\ref{lemma:orbitDynamicsRevisited}(3) implies $u_2 = e$ and $t u_1 h = h$. Since $U_{\bar{a}}$ is the trivial subgroup, $u_1 = e$ and so $th=h$. We conclude $b = t \in T_h$.

One has, by \cite[Proposition 1.11]{B09}, that for any algebraic group $H$ and $H$-variety $X$, the orbit $H \cdot x$, $x \in X$, is a subvariety of $X$ of dimension $\dim H - \dim H_x$. Thus
\begin{align*}
  \dim(\Xi) & = \dim(B_{L_I}) - \dim((B_{L_I})_h) \\
  & =  \dim(B_{L_I}) - \dim(T_h) \\
  & = (\ell(w_0(I)) + \dim(T)) - (\dim(T) - \dim(\Theta)) \\
  & = \ell(a) + \dim(\Theta),
\end{align*}
where the fourth equality is Lemma~\ref{lemma:equivalentCharUbara}(5). Hence
\begin{align*}  
  \dim(X_w^{\circ}) - \dim(\Xi) & = (\ell(a) + \ell(d)) - (\ell(a) + \dim(\Theta)) \\
  & = \ell(d) - \dim(\Theta) \\
  & = \dim(\Heart_I(X_w)) - \dim(\Theta),
\end{align*}
where the first equality is the length additivity of $ad$, while the final equality follows from Proposition~\ref{proposition:tStableHeartIso}. 
\end{proof}

We are now able to prove Theorem~\ref{theorem:torusLeviBorelBijection}.

\noindent \textit{Proof of Theorem~\ref{theorem:torusLeviBorelBijection}} Set $a = {_Iw}$ and $d = \prescript{I}{}{w}$ to match the notation we have been using throughout this section. Then \cite[Lemma 3.2]{GHY24} tells us that the set of $T$-orbits in $X^{\circ}_{d}$ for the action $\bigdot_{a^{-1}}$ is identical to the set of $T$-orbits in $X^{\circ}_{d}$ for the action $\bigdot$. Hence, Proposition~\ref{proposition:tStableHeartIso} implies that the map $\mathfrak{T}:\mathcal{O}_T(X^{\circ}_{d}) \rightarrow \mathcal{O}_T(\Heart_I(X_w))$ given by $\Theta\mapsto T a x$, where $x$ is any point in $\Theta$, is a codimension preserving bijection. The map $\mathfrak{O}$ is precisely the composition of a surjection and a bijection, namely $\mathfrak{O} = \beta \circ \mathfrak{T}$. Hence $\mathfrak{O}$ is itself a surjection. In the case where $L_I$ acts on $X_w$, $\mathfrak{O}$ is the composition of two codimension-preserving bijections, and hence is a codimension preserving bijection. \qed

\begin{corollary}
\label{corollary:bigCellBIComplexity}
If $L_I$ acts on $X_w$, the minimal codimension of a $B_{L_I}$-orbit in $X^{\circ}_w$ equals the torus complexity of $X^{\circ}_d$ for the usual torus action.
\end{corollary}

\subsection{Orbits in the Schubert Variety} \label{subsec:orbitsinSchubertVariety}
In this section let $w\in W$ and $I \subseteq \Delta$. For any $u \in W$, let $u = { _Iu} \prescript{I}{}{u}$ be the left parabolic decomposition of $u$ with respect to $I$. Our goal is to extend Corollary~\ref{corollary:bigCellBIComplexity} to a formula for the Levi complexity of a Schubert variety. To do this, we will need to lower bound the codimension of a $B_{L_I}$-orbit in $X^{\circ}_u$ for $u \leq w$. The added complication in this case is that $L_I$ may not act on $X_u$, though of course $B_{L_I}$ acts, and so $U_{\bar{a}}$ will not be the trivial subgroup of $G$.

For a subvariety $X$ of a variety $Y$, we write $\codim_{Y}(X)$ for the codimension of $X$ in $Y$.
\begin{proposition}
\label{proposition:BLICodim}
Let $u \leq w \in W$. Let $\Xi \in \mathcal{O}_{B_{L_I}}(X^{\circ}_u)$. Let $h \in \Xi \cap \Heart_I(X_u)$ and $\Theta = T \cdot h$. Then
\begin{align*} \codim_{X^{\circ}_u}(\Xi) &\geq \ell(u) - \supp(\prescript{I}{}{u}) - \ell(w_0(I))
\end{align*}
and
\begin{align*} \codim_{X_w}(\Xi) \geq \ell(w) - \supp(\prescript{I}{}{u}) - \ell(w_0(I)).
\end{align*}
\end{proposition}
\begin{proof}
The second claimed inequality follows from the first combined with the fact that the codimension of $X^{\circ}_u$ in $X_w$ is $\ell(w)-\ell(u)$. Hence it remains to prove the first inequality.
By \cite[Proposition 1.11]{B09} $\dim(\Xi) = \dim(B_{L_I}) - \dim((B_{L_I})_h) \leq \dim(B_{L_I}) - \dim(T_h)$. Multiplying this inequality by $-1$ and adding  $\ell(u)$ yields
\begin{align}
\label{eq:codimlb1}
\begin{split}
\ell(u) - \dim(\Xi) &\geq \ell(u) - \dim(B_{L_I}) + \dim(T_h)\\
                    &= \ell({ _Iu}) + \ell(\prescript{I}{}{u}) - (\dim(T) + \ell(w_0(I))) + \dim(T_h) \\
                    &= \ell(\prescript{I}{}{u}) - (\dim(T) - \dim(T_h)) - (\ell(w_0(I)) - \ell({ _Iu})) \\
                    &= \ell(\prescript{I}{}{u}) - \dim(\Theta) - (\ell(w_0(I)) - \ell({ _Iu})),
\end{split}
\end{align}
where the first equality is the length additivity of $u = { _Iu} \prescript{I}{}{u}$ and the last equality is \cite[Proposition 1.11]{B09}. The isomorphism in Proposition~\ref{proposition:tStableHeartIso} sends $T$-orbits to $T$-orbits and preserves their dimensions. Hence, $\dim(\Heart_I(X_u)) = \ell(\prescript{I}{}{u})$ and 
\begin{align}
\label{eq:codimlb2}
\ell(\prescript{I}{}{u}) - \dim(\Theta) = \codim_{\Heart_I(X_u)}(\Theta) \geq c_T(X^{\circ}_{\prescript{I}{}{u}}) \geq c_T(X_{\prescript{I}{}{u}}).
\end{align}
Thus,
\begin{align*}
\codim_{X^{\circ}_u}(\Xi) &= \dim(X^{\circ}_u) - \dim(\Xi) = \ell(u) - \dim(\Xi)\\
&\geq \ell(\prescript{I}{}{u}) - \dim(\Theta) - (\ell(w_0(I)) - \ell({ _Iu})) 
 \geq c_T(X_{\prescript{I}{}{u}}) - (\ell(w_0(I)) - \ell({ _Iu})) \\
&= \ell(\prescript{I}{}{u}) - \supp(\prescript{I}{}{u}) - (\ell(w_0(I)) - \ell({ _Iu})) 
= \ell(u) - \supp(\prescript{I}{}{u}) - \ell(w_0(I)),
\end{align*}
where the first inequality is \eqref{eq:codimlb1}, the second inequality is \eqref{eq:codimlb2}, the third equality is Corollary~\ref{cor:torus-complexity-Schubert}, and the final equality is the length additivity of $u = { _Iu} \prescript{I}{}{u}$.
\end{proof}

We are now ready to prove Theorem~\ref{theorem:generalTypeLeviComplexity}.
\begin{proof}[Proof of Theorem~\ref{theorem:generalTypeLeviComplexity}]
The equivalence of $L_I$ acting on $X_w$ and $I \subseteq \mathcal{D}_L(w)$ is Lemma~\ref{lemma:equivalentCharUbara}.

The Bruhat decomposition tells us that the Schubert variety $X_w$ is the disjoint union of the Schubert cells for $u \leq w$, $X_w = \bigsqcup_{u \leq w} X_u^{\circ}$, and hence 
\[
\mathcal{O}_{B_{L_I}}(X_w) = \bigsqcup_{u \leq w} \mathcal{O}_{B_{L_I}}(X^{\circ}_u).
\]
Thus
\begin{align*}
c_{L_I}(X_w) &= \min_{\substack{u \leq w \\ \Xi \in \mathcal{O}_{B_{L_I}}(X^{\circ}_u)}} \codim_{X_w}(\Xi)  \\
                 &\geq \min_{u \leq w}  \left( \ell(w) - \supp(\prescript{I}{}{u}) - \ell(w_0(I)) \right) \\
                 &= \ell(w)-\ell(w_0(I)) + \min_{u \leq w} \left(- \supp(\prescript{I}{}{u})\right) \\
                 &= \ell(\prescript{I}{}{w}) - \max_{u \leq w} \left(\supp(\prescript{I}{}{u}) \right) \\
                 &= \ell(\prescript{I}{}{w}) - \supp(\prescript{I}{}{w}),
\end{align*}
where the inequality is Proposition~\ref{proposition:BLICodim}, the third equality follows from Lemma~\ref{lemma:equivalentCharUbara}, and the final equality follows from the fact that $u \leq w$ implies $\prescript{I}{}{u} \leq \prescript{I}{}{w}$.

The locus of points with $T$-orbits of maximal dimension is dense and open in $X_{\prescript{I}{}{w}}$, so it meets the dense open Schubert cell $X^\circ_{\prescript{I}{}{w}}$. Consequently, $c_T(X^\circ_{\prescript{I}{}{w}})=c_T(X_{\prescript{I}{}{w}})$.

Corollary~\ref{corollary:bigCellBIComplexity} and Corollary~\ref{cor:torus-complexity-Schubert} imply that this lower bound on $c_{L_I}(X_w)$ is in fact achieved. We conclude that $c_{L_I}(X_w) = \ell(\prescript{I}{}{w}) - \supp(\prescript{I}{}{w})$. 
\end{proof}

\subsection{Complexity in the partial flag variety}
For $J\subset\Delta$, $G/P_J$ is called a \emph{partial flag variety}. For $w \in W^J$, the $B$-orbit of $w P_J$ is referred to as the \emph{Schubert cell} $X^{J, \circ}_w$; the closure of $X^{J, \circ}_w$ in $G/P_J$ with respect to the Zariski topology is the \emph{Schubert variety} $X^{J}_w$. When $J = \emptyset$, then $P_J = B$ and $G/B$ is the full flag variety. In this case, for $w \in W^{\emptyset} = W$ we shall continue to write $X_w^{\circ}$ and $X_w$ for the Schubert cell and variety in $G/B$, respectively. The goal of this section is to extend, where possible, the Levi-Borel complexity results of Section~\ref{subsec:orbitsinSchubertVariety} to Schubert varieties in the partial flag variety. 

We recall the relationship between Schubert varieties in the partial flag variety and Schubert varieties in the full flag variety. 
For any $w \in W^J$, the canonical projection $\pi_{B,P_J}: G/B \rightarrow G/P_J$ is $G$-equivariant and restricts to a birational morphism $\pi_{B,P_J}|_{X_w}: X_w \rightarrow X^{P_J}_w$. The inverse image is $\pi_{B,P_J}^{-1}(X^{P_J}_w) = X_{w w_0(J)}$.

\begin{lemma}\label{lemma:stabilizer} 
Let $w \in W^J$. Then ${\rm stab}_{G}(X^{P_J}_w) = P_{D_L(w w_0(J))}$.
\end{lemma}
\begin{proof}
We have ${\rm stab}_{G}(\pi_{B,P_J}^{-1}(X^{P_J}_w))\!=\!{\rm stab}_{G}(X_{w w_0(J)})\!=\!P_{D_L(w w_0(J))}$ by \cite[Lemma~8.2.3]{BL00}. Thus our result will follow if we show \[{\rm stab}_{G}(X^{P_J}_w)={\rm stab}_{G}(\pi_{B,P_J}^{-1}(X^{P_J}_w)).\]

We first show ${\rm stab}_{G}(X^{P_J}_w) \subseteq {\rm stab}_{G}(\pi_{B,P_J}^{-1}(X^{P_J}_w))$. Let $g \in {\rm stab}_{G}(X^{P_J}_w)$ and $x \in \pi_{B,P_J}^{-1}(X^{P_J}_w)$. Then $\pi_{B,P_J}(x) \in X^{P_J}_w$, and thus $g \in {\rm stab}_{G}(X^{P_J}_w)$ implies $g \pi_{B,P_J}(x) \in X^{P_J}_w$. The $G$-equivariance of $\pi_{B,P_J}$ yields $\pi_{B,P_J}(gx) \in X^{P_J}_w$, which implies $gx \in \pi_{B,P_J}^{-1}(X^{P_J}_w)$. We conclude $g \in {\rm stab}_{G}(\pi_{B,P_J}^{-1}(X^{P_J}_w))$.

Now we verify  ${\rm stab}_{G}(X^{P_J}_w) \supseteq {\rm stab}_{G}(\pi_{B,P_J}^{-1}(X^{P_J}_w))$. Let $g \in {\rm stab}_{G}(\pi_{B,P_J}^{-1}(X^{P_J}_w))$, $x \in X^{P_J}_w$, and $z \in \pi_{B,P_J}^{-1}(x)$. Then $gx=g\pi_{B,P_J}(z)=\pi_{B,P_J}(gz)$ by the $G$-equivariance of the map. Now since $g \in {\rm stab}_{G}(\pi_{B,P_J}^{-1}(X^{P_J}_w))$, we have $gz\!\in\!\pi_{B,P_J}^{-1}(X^{P_J}_w)$. As desired, we have $gx = \pi_{B,P_J}(gz) \in X^{P_J}_w$ and thus $g \in {\rm stab}_{G}(X^{P_J}_w)$.
\end{proof}

\begin{theorem} 
\label{thm:sphericalityTransfer}
Let $w \in W^J$ and $I \subseteq D_L(w w_0(J))$. Then $L_I$ acts on $X^{P_J}_w$. If $L_I$ also acts on $X_w$ (equivalently if $I \subseteq D_L(w)$), then $c_{L_I}(X^{P_J}_w) = c_{L_I}(X_w)$.
\end{theorem}
\begin{proof} 
The Levi subgroup $L_I$ acts on $X^{P_J}_w$ by Lemma~\ref{lemma:stabilizer}. If $I \subseteq D_L(w)$, then $L_I$ acts on $X_w$. Thus, since $\pi_{B,P_J}: G/B \rightarrow G/P_J$ is $G$-equivariant, we have that $\pi_{B,P_J}|_{X_w}: X_w \rightarrow X^{P_J}_w$ is an $L_I$-equivariant birational morphism. Since an $L_I$-equivariant birational map identifies dense $L_I$-stable open subsets, it preserves the codimension of a generic $B_{L_I}$-orbit, and hence the $L_I$-complexity. Thus $c_{L_I}(X^{P_J}_w) = c_{L_I}(X_w)$.
\end{proof}
We highlight that this \textit{does not} reduce the Levi complexity of Schubert varieties in the partial flag variety to that of the full flag variety. This reduction only works if $L_I$ acts on both $X^{P_J}_w$ and $X_w$, and importantly ${\rm stab}_{G}(X_w)$ is only a subgroup of ${\rm stab}_{G}(X^{P_J}_w)$. However, this does resolve the case of $L_I = T$, since $T$ acts on any Schubert variety in any partial flag variety.

\begin{corollary}\label{cor:toricClassification}
Let $w \in W^J$. Then $c_{T}(X^{P_J}_w) = c_{T}(X_w)$. In particular, $X^{P_J}_w$ is a toric variety if and only if each $s_j \in S$ appears at most once in any reduced expression of $w$, or equivalently, $\ell(w) = \supp(w)$.
\end{corollary}
\begin{proof}
Let $I = \emptyset$. Then $L_I = T$ and $I \subseteq D_L(w w_0(J))$ and $I \subseteq D_L(w)$. Thus, by Theorem \ref{thm:sphericalityTransfer}, $c_{T}(X^{P_J}_w) = c_{T}(X_w)$. Since $X^{P_J}_w$ is normal~\cite{RR85}, it is toric if and only if its torus complexity is zero. The final assertion follows from Corollary~\ref{cor:torus-complexity-Schubert}.
\end{proof}

\section*{Acknowledgements}
We thank Alexander Yong for many helpful conversations throughout the development of this paper.

\printbibliography

\end{document}